\documentclass[12pt]{amsart}
\usepackage{amssymb}
\usepackage{mathscinet}
\usepackage{cite}
\usepackage{enumitem}
\usepackage[margin=1.25in]{geometry}
\usepackage{latexsym,graphicx}
\usepackage[utf8]{inputenc}
\usepackage[mathscr]{euscript}
\usepackage{bbm}
\usepackage{amsmath}
\usepackage{amscd}
\usepackage{amsthm}
\usepackage{todonotes}
\usepackage{verbatim}
\usepackage{xcolor}
\usepackage[
	colorlinks,
	linkcolor={black},
	citecolor={black},
	urlcolor={black}
]{hyperref}

\def\freq{\varphi}
\def\Acoef{A}
\def\Bcoef{B}
\def\th{\vartheta}

\newcommand{\rmd}{{\mathrm{d}}}
\newcommand{\A}{{\mathcal A}}

\newcommand{\Hd}{{\mathrm{H}}}
\newcommand{\loc}{{\mathrm{loc}}}
\newcommand{\idty}{{\mathbf{I}}}

\newcommand{\set}[1]{\left\{#1\right\}}
\newcommand{\eqdef}{\overset{\mathrm{def}}=}

\date{}

\newcommand{\Z}{{\mathbb Z}}
\newcommand{\R}{{\mathbb R}}
\newcommand{\C}{{\mathbb C}}
\newcommand{\N}{{\mathbb N}}
\newcommand{\T}{{\mathbb T}}

\newcommand{\SL}{{\mathrm{SL}}}

\newtheorem{theorem}{Theorem}[section]
\newtheorem{lemma}[theorem]{Lemma}
\newtheorem{prop}[theorem]{Proposition}

\theoremstyle{definition}
\newtheorem{definition}[theorem]{Definition}
\newtheorem{remark}[theorem]{Remark}
\newtheorem{claim}{Claim}[theorem]
  
\newenvironment{claimproof}[1][Proof of Claim]{\noindent \underline{#1.} }{\hfill$\diamondsuit$}

\DeclareMathOperator{\tr}{Tr}
\DeclareMathOperator{\spectrum}{spec}
\newcommand{\traceMap}{{\mathsf{T}}}
\newcommand{\transferMatrix}{{\mathbf{T}}}
\newcommand{\monodromy}{{\mathbf{M}}}

\numberwithin{equation}{section}

\definecolor{purple}{rgb}{.5,0,1}
\definecolor{orange}{rgb}{1,.5,0}
\definecolor{green}{rgb}{0,.5,0}

\author[D. Damanik]{David Damanik}
\address{Department of Mathematics, Rice University, 6100 S. Main Street, Houston, Texas
77005-1892, USA}
\email{damanik@rice.edu}
\thanks{D.D.\ was supported in part by National Science Foundation grants DMS--2054752 and DMS--2349919}

\author[M. Embree]{Mark Embree}
\address{Department of Mathematics, Virginia Tech, Blacksburg, VA 24061, USA}
\email{embree@vt.edu}
\thanks{M.E.\ was supported in part by National Science Foundation grant DMS--2411141 and the Simons Institute for the Theory of Computing at UC Berkeley. }

\author[J. Fillman]{Jake Fillman}
\address{Department of Mathematics, Texas A\&M University, College Station, TX 77843,
USA}
\email{fillman@tamu.edu}
\thanks{J.\ F.\ was supported in part by National Science Foundation grant DMS--2513006 and by Simons Foundation grant MPS TSM--00013720.}

\author[A. Gorodetski]{Anton Gorodetski}
\address{Department of Mathematics, University of California, 
Irvine, CA 92697, USA}
\email{asgor@uci.edu}
\thanks{A.G. was supported in part by National Science Foundation grant DMS--2247966.}

\author[M. Mei]{May Mei}
\address{Department of Mathematics, Denison University, Granville, OH 43023, USA}
\email{meim@denison.edu}

\begin{document}

\title[Strongly Coupled Continuum Fibonacci Operators]{Continuum Fibonacci Schr\"odinger Operators\\ in the Strongly Coupled Regime}

\begin{abstract}
We study Schr\"{o}dinger operators on the real line whose potentials are generated by the Fibonacci substitution sequence and a rule that replaces symbols by compactly supported potential pieces. We consider the case in which one of those pieces is identically zero, and study the dimension of the spectrum in the large-coupling regime. 
Our results include a generalization of theorems regarding explicit examples that were studied previously and a counterexample that shows that the na\"ive generalization of previously established statements is false. 
In particular, in the aperiodic case, the local Hausdorff dimension of the spectrum does not necessarily converge to zero uniformly on compact subsets as the coupling constant is sent to infinity.
\end{abstract}

\maketitle

\hypersetup{
	linkcolor={black!30!blue},
	citecolor={black!30!green},
	urlcolor={black!30!blue}
}

\section{Introduction}

\subsection{Main Goals}
Since their discovery in the early 1980s, quasicrystals---structures admitting long-range order \emph{sans} periodicity---have played a significant role in materials science, physics, and mathematics.
The absence of periodicity coupled with long-range order leads to Hamiltonians with exotic spectral properties, such as singular continuous spectral measures, fractal spectrum, and anomalous quantum dynamical transport.

The Fibonacci substitution sequence has served as a central model of a one-dimensional quasicrystal starting with early works in the physics literature such as \cite{KohSutTan1987, KohKadTan1983PRL, Koh1992JSP, OPRSS1983}, and followed by a substantial amount of work in the mathematics literature, including \cite{Cantat2009DMJ, Cas1986CMP, DEGT08, DG11, DG12, DG13, DamGorYes2016Invent, Suto1989JSP, S87}; see \cite[Chapter~10]{DF2024ESO2} for more discussion and background about the Fibonacci Hamiltonian and \cite{BaaGri2013AO1} for more background on the mathematics of aperiodic order.
Changing the frequency from the inverse of the golden mean to other irrational numbers yields other models where the arithmetic properties of the continued fraction play an important role
\cite{Girand2014Nonlin, Mei2014JMP, MeiYes2014MMNP, Yessen2013JST}.
These works all studied the \emph{discrete} tight-binding model in $\ell^2(\Z)$. Here, we turn our attention to the continuum model in $L^2(\R)$ given by
\[ [H_V\psi] = -\psi''+V\psi \]
with $V: \R \to \R$ equivariant with respect to the Fibonacci sequence.
The continuum Fibonacci Hamiltonian is expected to share many features with the (by now) very well-understood discrete model, although it is considerably more subtle, due to its unbounded nature and the energy dependence of the Fricke--Vogt invariant.
This class of operators and other related models generated by aperiodic subshifts were first studied in the 2010s \cite{DFG2014, LenSeiSto2014JDE}, but many important questions remain open.
The work \cite{DFG2014} studied the dimension of the spectrum for \emph{constant} potential pieces in the regimes of high energy, small coupling, and high coupling.
{In \cite{DFG2014} the authors asked whether the asymptotics for the spectral dimension that they derived in the case of constant potential pieces remain
true for general pieces.
The results for large energies and small coupling were   shown to hold in complete generality (i.e., for arbitrary pieces of potential) \cite{FillmanMei2018AHP}, but the large-coupling regime remained elusive for the last decade.
The main objective of this work is to show that the asymptotics for high coupling constant from \cite{DFG2014} do \emph{not} hold for arbitrary potential pieces, while at the same time giving a partial result under a suitable positivity assumption.}
As we will see from the arguments and examples, the case of large coupling is considerably more delicate in general than the small-coupling and high-energy regimes.

Indeed, this is the key challenge in the work, especially compared to the discrete case, for which a more global understanding has been achieved \cite{DamGorYes2016Invent}.
More specifically, viewing $H_V = H_0+V$ where $H_0 = - \rmd^2/\rmd x^2$, one can recontextualize $H_{\lambda V}$ as a rescaling of $\lambda^{-1} H_0 + V$.
Thus, when $\lambda$ is large, the coefficient in front of $H_0$ is small.
In the discrete setting, $H_0$ is bounded and hence $\lambda^{-1} H_0 + V$ can be profitably studied as a perturbation of $V$, a diagonal operator.
However, in the present \emph{continuum} setting, $H_0$ and therefore $\lambda^{-1}H_0$ is unbounded and hence cannot be viewed as a small change to $V$, regardless of the size of $\lambda$.
Indeed, it seems that large-coupling asymptotics for continuum Schr\"odinger operators with ergodic potentials are quite delicate in general,\footnote{Indeed, the discrete case is well-studied, see, for instance, \cite{ACS83, MartinelliMicheli1987JSP, ShamisSpencer2015CMP, SoretsSpencer1991CMP} for results in that setting.} though there are a few positive results such as \cite{Bjerklov2006JMP, Bjerklov2007AHP, SoretsSpencer1991CMP}.
In fact, one crucial input needed for our work was a study of the large-coupling behavior of the Lyapunov exponent and integrated density of states (in the guise of the rotation number) for \emph{periodic} continuum Schr\"odinger operators, and even those results appear to be novel.

\subsection{Results}
An instance of the Fibonacci substitution sequence can be written as
\begin{equation}\label{eq:fib}
    \omega_n = 
    \left\lfloor (n+1)\alpha + \theta \right\rfloor
    - \left\lfloor n\alpha + \theta \right\rfloor,
    \quad n \in \Z,
\end{equation}
where $\alpha = (\sqrt{5}-1)/2$ denotes the inverse of the golden mean and $\theta \in \R$.
To produce the related continuum model, we choose functions $f_0,f_1 \in L^2([0,1))$ and define the resultant potential by placing a translated copy of $f_j$ in $[n,n+1)$ whenever $\omega_n=j$ ($j=0,1$). 
More precisely, we consider $V_\omega = V_{f_0,f_1,\omega}$ given by
\begin{equation} \label{eq:fibonacciPotDef}
    V_{\omega}(x) 
    = \sum_{ n \in \Z} f_{\omega_n}(x-n),
\end{equation}
where we slightly abuse notation by extending $f_j$ to vanish outside $[0,1)$.
One can also construct such potentials whenever $f_0$ and $f_1$ are defined on intervals of different lengths via 
 a suitable concatenation procedure; 
 this is discussed, for example, in  \cite[Section~1.2]{FillmanMei2018AHP}.

We denote the spectrum by $\Sigma = \Sigma(f_0,f_1)$, which is independent of the choice of $\theta \in \T$ \cite{DFG2014}.
In \cite{DFG2014}, the authors considered the case of locally constant potentials $f_0 = 0 \cdot \chi_{[0,1)}$ and $f_1 = 1 \cdot \chi_{[0,1)}$.
Abbreviating $\Sigma_\lambda := \Sigma(\lambda f_0,\lambda f_1)$, it was shown \cite[Corollary~6.6]{DFG2014} that
for this choice of  potential pieces, the local Hausdorff dimension of the spectrum satisfies
\begin{align}
\lim_{\lambda \to 0} \; \inf_{E \in \Sigma_\lambda} \; \dim_\Hd^\loc(E,\Sigma_\lambda) & = 1, \\
\lim_{K \to \infty} \; \inf_{E \in \Sigma_\lambda \cap [K,\infty)} \; \dim_\Hd^\loc(E, \Sigma_\lambda) & = 1.
\end{align}
Both of these statements were later generalized  to an arbitrary pair of functions $f_0, f_1$ such that the  {potential given by \eqref{eq:fib}--\eqref{eq:fibonacciPotDef}} is aperiodic\footnote{The statements are trivially true when the resulting potentials are periodic, i.e., $f_1=f_0$.} \cite[Theorems~1.3 and~1.4]{FillmanMei2018AHP}.

In the case $f_0 = 0 \cdot \chi_{[0,1)}$ and $f_1 =1\cdot \chi_{[0,1)}$, one also has (see \cite[Corollary 6.7]{DFG2014}): 
\begin{equation}\label{e.main}
\text{for any compact } \ S \subseteq \R, \ \text{\rm{we\ have}} \ 
\lim_{\lambda \to \infty} \dim_\Hd(\Sigma_\lambda \cap S) = 0.
\end{equation}
Such dimensional statements are quite useful, since they connect to questions of quantum dynamics \cite{Last1996JFA} and enable one (generally with even more work) to obtain information about higher-dimensional models \cite{DFG2021JFA}.

The goal of this manuscript is to study whether \eqref{e.main} can be generalized, and under what additional conditions this statement could hold. 
Our first main result shows that the natural extension of previous work to arbitrary potential pieces is \emph{false} in a fairly drastic fashion.
Let us denote by $C_0([0,1])$ the continuous functions $[0,1] \to \R$ satisfying $f(0)=f(1)=0$ (and note that this produces $V_\omega$ in \eqref{eq:fibonacciPotDef} that are \emph{continuous}).

\begin{theorem}\label{t:counterexample}
There exist $f_0 \neq f_1 \in C_0([0,1])$, $E \in \R$, and $\lambda_k \uparrow \infty$ such that
\begin{equation} \label{eq:main:counterex}
E \in \Sigma_{\lambda_k}\text{ and } \dim_\Hd^\loc(\Sigma_{\lambda_k},E) =1 \quad
\forall k.
\end{equation}
\end{theorem}

\begin{remark}\mbox{\,}
\begin{enumerate}[label={\rm(\alph*)}]
    \item 
    Let us reiterate that the assumption $f_0 \neq f_1$ implies that the resulting potentials are aperiodic and hence the content of the theorem is nontrivial.
    \item The term ``pseudo band'' for the points in a spectrum of zero measure that have full local Hausdorff dimension was suggested in \cite{BaaJosKra1992PLA}, where the Kronig-–Penney model for the Fibonacci potential was considered; see also \cite{KOLLAR1986203} for related results. In this sense the energies described in Theorem \ref{t:counterexample} can be considered as ``pseudo bands'' that persist for arbitrarily large values of the coupling constant.    
\end{enumerate}

\end{remark}

For any compact set $S$ containing $E$ in its interior, \eqref{eq:main:counterex} implies $\dim_\Hd(\Sigma_{\lambda_k} \cap S) = 1$ for any $k$, so, in particular, \eqref{e.main} fails. 
Nevertheless, a partial result holds under a suitable sign-definiteness assumption.

\begin{theorem} \label{t:main:nwd}
Suppose $f_0 = 0 \cdot \chi_{[0,1)}$, $f_1$ is continuous and nonnegative, and $\{x : f_1(x)=0\}$ is nowhere dense in $[0,1]$. 
In this case, \eqref{e.main} holds.    
\end{theorem}

Figure~\ref{fig:f1pos} shows the lower portion of the (unbounded) spectrum of periodic approximations to two continuum Fibonacci operators.  In both cases $f_0 = 0\cdot \chi_{[0,1)}$.  On the left, $f_1 = 1 \cdot \chi_{[0,1)}$; on the right $f_1(x) = \exp(1+1/((2x-1)^2-1))$, a $C^\infty$ bump function that gives a continuous potential throughout $\R$. 
(We discuss the numerical computations used to create these plots in Section~\ref{sec:nlevp}.)

\begin{figure}[b!]
\includegraphics[width=2.85in]{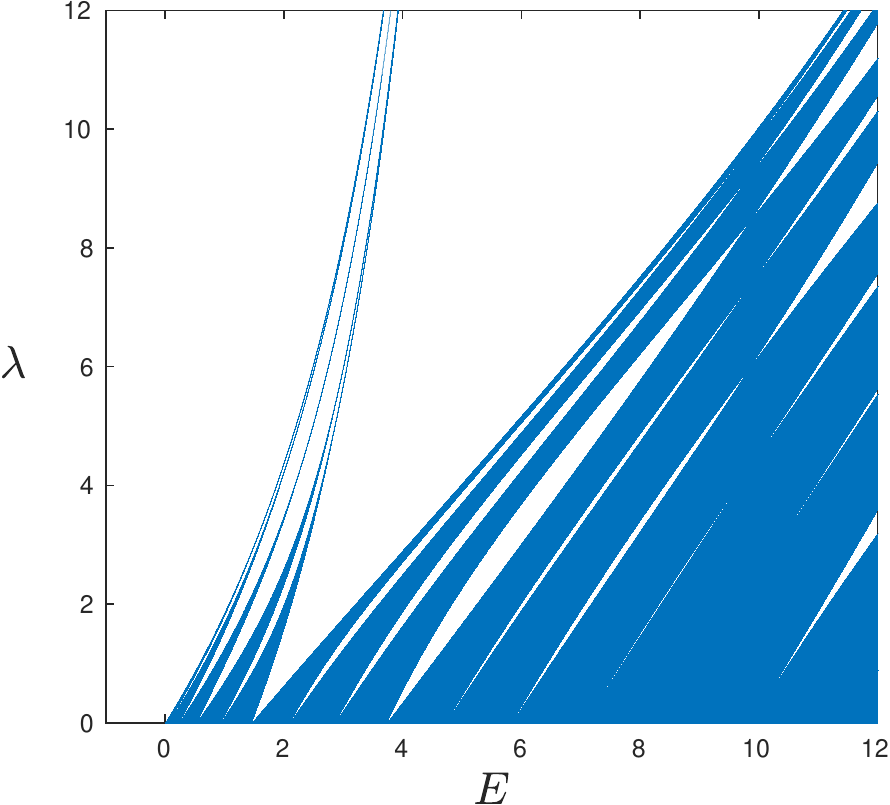}\quad
\includegraphics[width=2.85in]{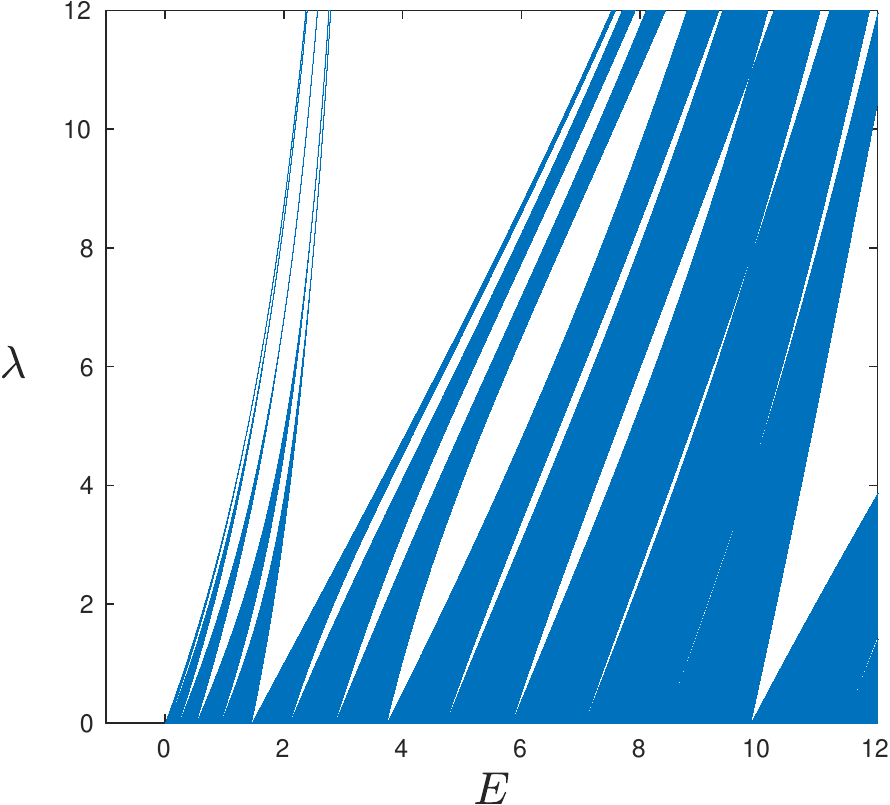}
\caption{\label{fig:f1pos}The lower portion of the spectrum for periodic approximations (period $p=13$) for two continuum Fibonacci operators with $f_0 = 0\cdot \chi_{[0,1)}$, with $f_1$ constant (left) and a $C^\infty$ bump function that is positive for $x\in(0,1)$ (right).  For each value of $\lambda$, the corresponding spectrum is a horizontal slice of the plot.}
\end{figure}

\begin{remark}\label{r.first}
The case in which the set $\{f_1 =0\}$ contains some intervals seems to be more complicated, and if \eqref{e.main} holds in that case, the proof will require some other arguments. 
Indeed, it is not hard to see that in this case the conclusion of Lemma~\ref{lem:traceAsymptoticsLargelambda} no longer holds true for $\varphi = f_1$.
\end{remark}

We give the proofs of the main theorems in Section~\ref{sec:proofs} after recalling some background in Section~\ref{sec:bg}.
The construction of a (family of) counterexamples is quite delicate and requires some surprisingly challenging asymptotics for the Lyapunov exponent and rotation number for the case of periodic Schrodinger operators in the large-coupling regime.
We expect that the asymptotics worked out for this proof may be of interest independent of the main results of the current work.
We need a particular fact about expanding directions of $\SL(2,\R)$ matrices, which we recall in Appendix~\ref{sec:funfact}.\ \ 
In Section~\ref{sec:nlevp} we briefly review Floquet theory to characterize the spectra of periodic approximations, and then show how, for piecewise constant potentials, such spectra can be described as solutions to a parameter-dependent, finite-dimensional nonlinear eigenvalue problem.

\section*{Acknowledgments}  
We thank the American Institute of Mathematics for their support and hospitality during a recent SQuaRE meeting, at which some of this work was done.

\section{Background}\label{sec:bg}

\subsection{Subshifts}
Let $\mathcal{A}$ be a finite set, called the \emph{alphabet}. Equip $\mathcal{A}$ with the discrete topology and endow the \emph{full shift} $\mathcal{A}^\Z := \{(\omega_n)_{n \in \Z} :   \omega_n \in \A \text{ for all } n \in \Z\}$ with the corresponding product topology. The \emph{shift}
\[
[T \omega]_n \eqdef \omega_{n+1},
\quad \omega \in \A^{\Z}, \; n \in \Z,
\]
defines a homeomorphism from $\A^{\Z}$ to itself. A subset $\Omega \subseteq \mathcal{A}^\Z$ is called \emph{$T$-invariant} if $T^{-1} \Omega  = \Omega$. 
Any compact $T$-invariant subset of $\A^\Z$ is called a \emph{subshift}.

We can associate potentials (and hence Schr\"odinger operators) with elements of subshifts as follows. 
For each $\alpha \in \mathcal{A}$, we fix a real-valued function $f_\alpha \in L^2([0,1))$. 
Then, for any $\omega \in \A^\Z$, we define the action of the continuum Schr\"odinger operator $H_\omega$ in $L^2(\R)$ by
\begin{equation}\label{eq:Schrod}
H_\omega
=
- \frac{\rmd^2}{\rmd x^2}+ V_\omega,
\end{equation}
where the potential $V_\omega = V_{\{f_\alpha : \alpha \in \A\}, \omega}$ is given by
\begin{equation} \label{eq:VomegaDef}
V_\omega(x)
= \sum_{n\in\Z} f_{\omega_n}(x-n),
\end{equation}
where, as before, $f_\alpha$ is extended to vanish outside $[0,1)$.
These potentials belong to $L^2_\mathrm{loc, unif}(\R)$ and hence each $H_\omega$ defines a self-adjoint operator on a dense subspace of $L^2(\R)$ in a canonical fashion.

\subsection{The Fibonacci Subshift}

In this paper, we study a special case of the foregoing construction, namely potentials generated by elements of the Fibonacci subshift. In this case, the alphabet contains two symbols, $\A \eqdef \{ 0,1 \}$. The Fibonacci substitution is the map
\[
S(0) = 1, \; S(1) = 10.
\]
This map extends by concatenation to $\A^*$, the free monoid over $\A$ (i.e., the set of finite words over $\A$), as well as to $\A^{\N}$, the collection of (one-sided) infinite words over $\A$. There exists a unique element
\[
u
=
10110 101 10110 \ldots \in \A^{\N}
\]
with the property that $u = S(u)$. It is straightforward to verify that for $n \in \N$, $S^n(1)$ is a prefix of $S^{n+1}(1)$. Thus, one obtains $u$ as the limit (in the product topology on $\mathcal{A}^\mathbb{N}$) of the sequence of finite words $\{ S^n(1) \}_{n \in \N}$. 
The Fibonacci subshift consists of all two-sided infinite words with the same local factor structure as $u$, that is,
\[
\Omega
\eqdef
\{ \omega \in \mathcal{A}^\Z : \text{every finite subword of $\omega$ is also a subword of } u \}.
\]
The reader may notice that this appears to be a different paradigm than the one introduced in \eqref{eq:fib}, but rest assured that these two definitions are compatible.\footnote{The interested reader can consult, for example, \cite[Chapter~10]{DF2024ESO2} or \cite[Chapter~2]{Lothaire2002} for detailed explanations connecting the two perspectives.}
 Given real-valued functions $f_0, f_1 \in L^2([0,1))$, we consider the family of continuum Schr\"odinger operators $\{H_\omega\}_{\omega \in \Omega}$ defined by \eqref{eq:Schrod} and \eqref{eq:VomegaDef}. 
Since $(\Omega,T)$ is a minimal dynamical system, there is a uniform closed set $\Sigma = \Sigma(f_0,f_1) \subseteq \R$ with the property that
\[
\spectrum(H_\omega)
=
\Sigma
\text{ for every }
\omega \in \Omega.
\]

Of course, one can choose $f_0 = f_1$, implying that every $V_\omega$ is a periodic potential, in which case Floquet theory reveals the spectrum to be a union of nondegenerate closed intervals, and hence, in particular, \emph{not} nowhere dense.
The main result of \cite{DFG2014} is that periodicity of the potentials $V_\omega$ is the only possible obstruction to Cantor spectrum. 
We will later specialize to the case $f_0\equiv 0$.
In this case, to ensure that each $V_\omega$ is aperiodic, it suffices to insist that $f_1 \not\equiv 0$  in $L^2$ (i.e., $f_1$ does not vanish a.e.).

\begin{theorem}[{DFG \cite[Corollary~5.5]{DFG2014}}]
Let $\Omega$ denote the Fibonacci subshift over $\A = \{0,1\}$. If the potential pieces $f_0$ and $f_1$ are chosen so that $V_\omega$ is aperiodic for one $\omega \in \Omega$ {\rm(}hence for every $\omega \in \Omega$ by minimality{\rm)}, then $\Sigma$ is an extended Cantor set\footnote{That is, a closed (but not necessarily compact) perfect and nowhere dense set.} of zero Lebesgue measure.
\end{theorem}

Let us also inductively define $\ell_k>0$ and $f_{k} \in L^2([0,\ell_k))$ by
\begin{align} \label{eq:ellkDef}
\ell_0 & = \ell_1  = 1, \quad \ell_k = \ell_{k-1} + \ell_{k-2} \\
\label{eq:fkDef}
f_{k}(x)
& =
\begin{cases}
f_{k-1}(x), & 0 \leq x < \ell_{k-1}; \\
f_{k-2}(x-\ell_{k-1}), & \ell_{k-1} \leq x < \ell_{k}.
\end{cases}
\end{align}

\subsection{Trace Map, Invariant, and Local  Dimension of the Spectrum}\label{ss.definitions} 

The spectrum (and many spectral characteristics) of the continuum Fibonacci model can be encoded in terms of an associated polynomial diffeomorphism of $\R^3$, called the \emph{trace map}. 
Let us make this correspondence explicit, following \cite{DFG2014, FillmanMei2018AHP}. 

To begin, we need to set up some notation.
Consider the differential equation
\begin{equation} \label{eq:TimeIndepSchrod}
-y''(x) + f(x) y(x)
=
E y(x),\quad
E \in \C, \quad x \in I,
\end{equation}
where $I \subseteq \R$ is an interval and $f \in L^2_{\loc}(I)$.
Given $a \in I$, we write $u_{f,E,a},v_{f,E,a}$ for the solutions of \eqref{eq:TimeIndepSchrod} satisfying the initial conditions 
\begin{equation}
u_{f,E,a}(a) = v_{f,E,a}'(a)=1 \text{ and } u_{f,E,a}'(a) = v_{f,E,a}(a)=0.
\end{equation}
For $a,b \in I$, we then define the transfer matrix $\transferMatrix_{[a,b]}(f,E)$ by
\begin{equation} \label{eq:TransMatDef}
    \transferMatrix_{[a,b]}(f,E) \eqdef
    \begin{bmatrix}
        v'_{f,E,a}(b)
        & u'_{f,E,a}(b)  \\
         v_{f,E,a}(b) &  
        u_{f,E,a}(b)
    \end{bmatrix},
\end{equation}
and note that $\transferMatrix_{[a,a]}(f,E) = \idty$, $\transferMatrix_{[b,c]}(f,E)\transferMatrix_{[a,b]}(f,E)= \transferMatrix_{[a,c]}(f,E)$ (by existence and uniqueness of solutions of \eqref{eq:TimeIndepSchrod}), and $\det \transferMatrix_{[a,b]}(f,E) = 1$ for any $a,b,c,f,E$.

Returning to the Fibonacci setting, we define the monodromy matrices by
\begin{align*}
\monodromy_k(E)
& {\eqdef}
\transferMatrix_{[0,\ell_k]}(f_k,E),
\quad
k \in \Z_{\geq 0}, \; E \in \C,
\end{align*}
with $\ell_k$ and $f_k$ as in \eqref{eq:ellkDef}--\eqref{eq:fkDef}.
Their half-traces are denoted by
\[
x_{k}(E)
\eqdef
\frac{1}{2}
\tr(\monodromy_k(E))
=
\frac{1}{2}\left( u_{f_k,E,0}(\ell_k) + v'_{f_k,E,0}(\ell_k) \right),
\quad
k \in \Z_{\geq 0}, \; E \in \R.
\]
From the definitions (and existence and uniqueness of solutions of \eqref{eq:TimeIndepSchrod}), we note that
\[
\monodromy_{k}(E)
=
\monodromy_{k-2}(E)\monodromy_{k-1}(E),
\]
which, with the help of the Cayley--Hamilton theorem, implies
\begin{equation} \label{eq:traceRecursion}
    x_{k+1}  = 2x_kx_{k-1} - x_{k-2}
\end{equation}
for all relevant $k$.
We can encode this recursion by a polynomial map $\R^3 \to \R^3$ as follows: the \emph{trace map} is defined by
$$
\traceMap(x,y,z)
\eqdef
(2 xy-z,x,y),
\quad
x,y,z \in \R,
$$
and, in view of \eqref{eq:traceRecursion}, one has
\begin{equation}
    \traceMap^k(x_2,x_1,x_0)
    = (x_{k+2}, x_{k+1}, x_k),\quad
    k \geq 0.
\end{equation}
Thus, the function $\gamma:\R \to \R^3$ given by $\gamma(E) \eqdef (x_{2}(E), x_1(E), x_0(E))$ is known as the \emph{curve of initial conditions}.

The map $\traceMap$ is known to have a first integral given by the so-called \emph{Fricke--Vogt invariant}, defined by
$$
I(x,y,z) \eqdef x^2 + y^2 + z^2 - 2 xyz -1,
\quad
x,y,z \in \R.
$$
More precisely, $I\circ \traceMap = I$, so $\traceMap$ preserves the level surfaces of $I$:
$$
S_V
\eqdef
\{ (x,y,z) \in \R^3 : I(x,y,z) = V \},
\quad V \in \R.
$$
Consequently,
 every point of the form $\traceMap^k(x_{2}(E), x_1(E), x_0(E))$ with $k \in \Z_{\geq 0}$ lies on the surface $S_{I(\gamma(E))}$. 
 For the sake of convenience, we put
$$
I(E)
\eqdef
I(\gamma(E))
=
I(x_{2}(E), x_1(E), x_0(E)),
$$
with a minor abuse of notation. Putting everything together:
\begin{equation} \label{eq:tracemapInvariant}
   I(x_{k+1}(E), x_k(E), x_{k-1}(E)) = I(\gamma(E))  
\end{equation}
for every $E$ and every relevant $k$.

When $V<0$, the set $S_V$ has five connected components: one compact connected component that is diffeomorphic to the 2-sphere $S^2$, and four unbounded connected components, each of which is diffeomorphic to the open unit disk. 
When $V = 0$, each of the four unbounded components meets the compact component, forming four conical singularities. 
As soon as $V >0$, the singularities resolve; for such $V$, the surface $S_V$ is smooth, connected, and diffeomorphic to the four-times-punctured 2-sphere.

The trace map is important in the study of operators of the type \eqref{eq:Schrod}, as its dynamical spectrum, defined by
$$
B
\eqdef
\set{ E \in \R : \{ \traceMap^k(\gamma(E)) : k \in \Z_{\geq 0} \} \text{ is bounded} },
$$
encodes the operator-theoretic spectrum of $H$, which was first proved by S\"{u}t\H{o} in the discrete setting \cite{Suto1989JSP}.

\begin{prop}[{DFG \cite[Proposition~6.3]{DFG2014}}]
\label{p.sigmab}
$\Sigma = B$.
\end{prop}

There are several substantial differences between the continuum setting and the discrete setting. First, in the discrete case, the Fricke--Vogt invariant is constant (viewed as a function of $E$, $I = I(E)$). However, the invariant may enjoy nontrivial dependence on $E$ in the continuum setting, which is demonstrated by examples in \cite{DFG2014}. This dependence is related to new phenomena that emerge in the continuum setting and make its study worthwhile.
Moreover, we will show in Proposition~\ref{prop:IEnonconstant} that such dependence is an unavoidable feature of the continuum setting: as soon as the potentials are aperiodic, the function $I$ \emph{must} be nonconstant.

Second, the Fricke--Vogt invariant is always nonnegative in the discrete setting (even away from the spectrum), but one cannot \textit{a priori} preclude negativity of $I$ in the continuum setting. However, it is proved in \cite{DFG2014} that any energies for which $I(E) < 0$ must lie in the resolvent set of the corresponding continuum Fibonacci Hamiltonian.

\begin{prop}[{DFG \cite[Proposition~6.4]{DFG2014}}] \label{p.nonneginv}
For every $E \in \Sigma$, one has $I(E) \ge 0$.
\end{prop}

To study the fractal dimension of the spectrum, we will use the following theorem from \cite{DFG2014}, which relates local fractal characteristics near an energy $E$ in the spectrum to the value of the invariant at $E$.

\begin{theorem}[{DFG \cite[Theorem~6.5]{DFG2014}}] 
\label{t.fiblocaldim}
There exists a continuous map $D : [0,\infty) \to (0,1]$ that is real-analytic on $\R_+$ with the following properties:
\begin{itemize}

\item[{\rm (i)}] $\dim_\Hd^\loc(\Sigma;E) = D(I(E))$ for each $E \in \Sigma$;

\item[{\rm (ii)}] We have $D(0) = 1$ and $1 - D(I) \sim \sqrt{I}$ as $I \downarrow 0$;

\item[{\rm (iii)}] We have
$$
\lim_{I \to \infty} D(I) \cdot \log I = 2 \log (1 + \sqrt{2}).
$$

\end{itemize}
\end{theorem}
\noindent Thus, to study the local fractal dimensions of the spectrum, it suffices to understand the invariant $I$.

Let us return to one of the difficulties mentioned above: in the current setting the invariant can in principle be nonconstant (as a function of the energy).
In fact, not only is it possible that $I$ is nonconstant; it is nonconstant if and only if $f_0 \neq f_1$ (i.e., the potentials $V_\omega$ are aperiodic).

\begin{prop} \label{prop:IEnonconstant}
With setup as above,    $I(E)$ is constant if and only if $f_0=f_1$, in which case it vanishes identically.
\end{prop}

\begin{proof}
With the help of the Cayley--Hamilton theorem, one can check that
\begin{equation} \label{eq:FVIintermsofCommutator}
    I(\tfrac12 \tr \mathbf{A},
    \tfrac12 \tr \mathbf{B},
    \tfrac12 \tr \mathbf{AB})
    = \tfrac14(\tr(\mathbf{A}^{-1}\mathbf{B}^{-1}\mathbf{AB}) - 2), \quad \mathbf{A}, \mathbf{B} \in \SL(2,\R).
\end{equation}
In particular, for $\mathbf{A}, \mathbf{B} \in \SL(2,\R)$,
\begin{equation} \label{eq:FVIvanishIFFcommute}
        I(\tfrac12 \tr \mathbf{A}, \tfrac12 \tr \mathbf{B}, \tfrac12\tr \mathbf{AB}) = 0
        \text{ whenever }  \mathbf{AB}=\mathbf{BA}.
    \end{equation}
On one hand, if $f_0=f_1$, then  $\monodromy_0(E) = \monodromy_1(E)$, which certainly commute for every $E$, leading to $I \equiv 0$.

On the other hand, if $I$ is constant, then, on account of \cite[Section~3]{FillmanMei2018AHP}, that constant must be zero.
From here, the proof that $f_0=f_1$ is similar to and inspired by the proof of   \cite[Theorem~2.1]{BDFGVWZ2019JFA}.
To keep this paper self-contained, we give the details. Consider $E \in \R$. 
Due to \eqref{eq:FVIintermsofCommutator}, it follows that $\tr \big( \monodromy_0(E)^{-1} \monodromy_1(E)^{-1} \monodromy_0(E) \monodromy_1(E) \big) = 2$, which (since all matrices in question belong to $\SL(2,\R)$) implies that 
\[
[\monodromy_0(E), \monodromy_1(E)]
\eqdef \monodromy_0(E)\monodromy_1(E) - \monodromy_1(E)\monodromy_0(E)
\]
has a nontrivial kernel and therefore $\monodromy_0(E)$ and $\monodromy_1(E)$ have at least one common eigenvector by a standard argument in linear algebra (cf.\ the proof of \cite[Theorem~40.5]{Prasolov1994linear}). Now, for every $E$ for which $\monodromy_0$ is elliptic, its eigenvectors can be chosen to be complex conjugates of one another and hence $\monodromy_0(E)$ and $\monodromy_1(E)$ are simultaneously diagonalizable for all such $E$.\ \ 
Consequently, 
\begin{equation}
    \monodromy_0(E)\monodromy_1(E) - \monodromy_1(E)\monodromy_0(E)
\end{equation}
is an analytic function of $E$ that vanishes on a nondegenerate closed interval, hence vanishes identically.

Equivalently, $\transferMatrix_{[0,2]}(f_0\star f_1,E) \equiv \transferMatrix_{[0,2]}(f_1 \star f_0,E)$, where we write $f_0 \star f_1$ for the \emph{concatenation} of $f_0$ and $f_1$ (compare \eqref{eq:fkDef}).
By the Borg--Marchenko theorem (cf.\ \cite{Bennewitz, Borg, Marchenko}), 
we have $f_0 \star f_1 \equiv f_1 \star f_0$, whence $f_0=f_1$.
\end{proof}

\section{Proofs of Main Results}
\label{sec:proofs}

Let us begin with a preparatory result about nonnegative potentials in the large-coupling regime.
In the lemma below, we emphasize that the conclusion holds as long as the potential $\varphi$  does not vanish on any nontrivial interval (however, it is otherwise permitted to vanish on a set of positive measure).

\begin{lemma} \label{lem:traceAsymptoticsLargelambda}
Suppose $\varphi:[a, b]\to [0,\infty)$ is continuous  and $\{x : \varphi(x)=0\}$ is nowhere dense in $[a,b]$. 
Let $\monodromy_{E, \lambda} = \transferMatrix_{[a,b]}(\lambda \varphi,E)$ be the monodromy matrix related to the equation
\begin{equation}\label{e.initial}
-y''(x)+\lambda\varphi(x)y(x)=Ey(x), \ x\in [a,b], 
\end{equation}
as defined in \eqref{eq:TransMatDef}. Then for any $K\in \R$ we have
$$
\tr \monodromy_{E,\lambda}\to \infty\ \text{as}\ \lambda\to \infty,
$$
uniformly in $E \leq K$.
In particular, $\monodromy_{E,\lambda}$ is hyperbolic for $\lambda$ large enough.
\end{lemma}

To prove Lemma~\ref{lem:traceAsymptoticsLargelambda}, we break some of the main comparison estimates into three further lemmata.

\begin{lemma} \label{lem:autCompare}
Suppose 
\[\mathbf{A}(x)=\begin{bmatrix}
  a_{11}(x) & a_{12}(x) \\
  a_{21}(x) & a_{22}(x) \\
\end{bmatrix}, \quad x \in [a,b]
\]
is a continuous matrix-valued function with nonnegative entries and that
\[\mathbf{B}=\begin{bmatrix}
  b_{11} & b_{12} \\
  b_{21} & b_{22} \\
\end{bmatrix}\]
is a constant matrix with nonnegative entries satisfying $b_{ij}\le a_{ij}(x)$, $i,j=1,2$ for all $x\in [a,b]$. Let $w\in \R^2$ be a nonzero vector with $w_1\ge 0, w_2\ge 0$, and consider solutions $y$ and $z$ of the Cauchy problems
\[  \frac{\rmd y}{\rmd x}
=\mathbf{A}(x)y, \qquad \frac{\rmd z}{\rmd x} = \mathbf{B}z, \qquad y(a)=z(a)=w. \] 
Then $y_i(x)\ge z_i(x)$, $i=1,2$, for all $x\in [a,b]$.     
\end{lemma}
 \begin{proof}
Consider without loss the case $a=0$. Notice that the Cauchy problem for $z$ is solved by $z(x) = e^{xB}w$ and that the Cauchy problem for $y-z$ can be written as
\[ 
\frac{\rmd}{\rmd  x}(y-z) = \mathbf{A}(x)(y-z)+ (\mathbf{A}(x)-\mathbf{B})z,\quad (y-z)(0)=0.
\]
By the assumptions on $\mathbf{A}(\cdot)$, $\mathbf{B}$, and $w$, it follows that $y(x) - z(x)$ has nonnegative entries for all $x$.
 \end{proof}

\begin{lemma} \label{lem:normgrowth}
Suppose $m \geq 1$ and let $y(x)= [y_1\ y_2]^\top$ be a solution of
\begin{equation} \label{eq:schrodingerFlowConstm}
\begin{bmatrix}
  y_1' \\
  y_2' \\
\end{bmatrix}=\begin{bmatrix}
 0 & m \\
  1 & 0 \\
\end{bmatrix}\begin{bmatrix}
  y_1 \\
  y_2 \\
\end{bmatrix}
\end{equation}
such that $y_1(a)\ge 0$ and $y_2(a)\ge 0$.
Then, 
\begin{equation}
    \frac{\rmd^2}{\rmd x^2}(y_1^2+y_2^2)\ge 2(1+m)(y_1^2+y_2^2).
\end{equation}
\end{lemma}

\begin{proof}
Denote $R = y_1^2+y_2^2$ and note from direct computations using \eqref{eq:schrodingerFlowConstm} that  
\begin{align}
\nonumber
R'' = 2\big[(y_1')^2+y_1^{}y_1'' + (y_2')^2 + y_2^{}y_2'')\big]
& = 2\big(m^2y_2^2 + my_1^2 + y_1^2+ m y_2^2 \big) \\
&   
    \geq 2(1+m)(y_1^2+y_2^2),
\end{align}
as promised.
\end{proof}

\begin{lemma} \label{lem:C2growth}
Suppose $k>0$ and $R(x)$, $x\in [a,b]$, is a $C^2$-function such that 
\begin{equation}
\begin{cases} 
    R'' \ge kR,  \\[1mm]
    R' \ge 0, \\[1mm]
    R(a)=R_0>0. 
\end{cases} 
\end{equation}
Then $R(x)\ge \frac{R_0}{2}\left(e^{\sqrt{k}(x-a)}+e^{-\sqrt{k}(x-a)}\right)$ for all $x\in [a,b]$.  
\end{lemma}
\begin{proof}
For a small $\varepsilon>0$, consider a related system $S'' = kS$ with initial conditions $S(a)=R_0-\varepsilon>0$ and $S'(a) =0$. Then we have
$$
S(x) = \frac{R_0-\varepsilon}{2}\left(e^{\sqrt{k}(x-a)} +
        e^{-\sqrt{k}(x-a)}\right).
$$
If $T(x)=R(x)-S(x)$, then we have $T''\ge kT$, $T'(a)\ge 0$, $T(a)=\varepsilon>0$, from which we see that $T(x)\ge \varepsilon>0$ for all $x\in [a,b]$. Therefore, 
$$
R(x)\ge \frac{R_0-\varepsilon}{2}\left(e^{\sqrt{k}(x-a)} +
        e^{-\sqrt{k}(x-a)}\right).
$$
Since $\varepsilon>0$ can be taken arbitrarily small, this implies the desired estimate.
\end{proof}

We now use these results to prove the first main technical lemma.

\begin{proof}[Proof of Lemma~\ref{lem:traceAsymptoticsLargelambda}]
Consider the nonautonomous flow on $\R^2$ given by
\begin{equation} \label{eq:nonautFlow}
    \frac{\rmd}{\rmd x}\begin{bmatrix}  y_1 \\ y_2 \end{bmatrix}=\begin{bmatrix}  0 & \lambda \varphi(x)-E\\  1 & 0\end{bmatrix} \begin{bmatrix}  y_1 \\ y_2 \end{bmatrix},
\end{equation}
and notice that $y$ solves \eqref{e.initial} if and only if $(y',y)^\top$ solves \eqref{eq:nonautFlow}.
Then the Dirichlet solution (that is, the solution that satisfies $y(a) = 0$, $y'(a) = 1$) corresponds to the dynamics of the vector $(1,0)^\top$, and the Neumann  solution ($y(a)=1$, $y'(a)=0$) corresponds to the dynamics of the vector $(0,1)^\top$. 

For any vector $u=(u_1, u_2)^\top \in \R^2$, and any $x_1<x_2$ from $[a,b]$, denote by 
\begin{equation} 
F_{[x_1, x_2]}(u) \eqdef  \transferMatrix_{[x_1,x_2]}(\lambda\varphi,E)u \in \R^2
\end{equation} 
the solution of  \eqref{eq:nonautFlow} at the moment~$x_2$, for initial conditions $y_1(x_1)=u_1, y_2(x_1)=u_2$.

For $t \in (-\pi/2,\pi/2)$, define the cone
$$
V_t \eqdef \left\{v\in \R^2\setminus\{0\} :  \arg v\in \left(0, \frac{\pi}{2} + t\right)\right\}.
$$
Fix $s>0$ small.
We will split the proof into several claims. 

\vspace{3pt}

\begin{claim} There exists $\varepsilon = \varepsilon(\varphi, K, s)>0$ such that for any {$E\leq K$} and any $\lambda\ge 0$, the following hold:
\begin{enumerate}[label={\rm(\alph*)}]
\item For any vector $w \in V_0$ and any $[x_1, x_2]\subseteq [a, b]$ with $|x_2-x_1|\le \varepsilon$, one has $F_{[x_1, x_2]}(w)\in V_s$.

\item For any vector $w\in V_{-s}$ and any $[x_1, x_2]\subseteq [a, b]$ with $|x_2-x_1|\le \varepsilon$, one has $F_{[x_1, x_2]}(w) \in V_0$.

\item In both settings, 
$|F_{[x_1, x_2]}(w)|>\frac{1}{2}|w|$.
\end{enumerate}
\end{claim}

\begin{claimproof} 
Assume without loss that $K \geq 1$.
Consider $w \in V_0$ and the solution $y = (y_1,y_2)^\top$ of \eqref{eq:nonautFlow} with $y(x_1)=w$.
Write $ y(x)= (y_1(x), y_2(x))^\top$ in polar coordinates as
\begin{equation}
    y_1(x) = r(x)\cos\theta(x), \quad y_2(x) = r(x) \sin\theta(x),
\end{equation} 
with $r \geq 0$ and $\theta$ chosen continuously in $x$ with $\theta(x_1) \in (0,\pi/2)$.
Note that
\begin{align} \label{eq:argderivative}
   \frac{\rmd \theta}{\rmd x}
    = \frac{y_1^{}y_2'- y_1'y_2^{}}{y_1^2 + y_2^2}
    = \frac{y_1^2- (\lambda \varphi -E) y^2_2}{y_1^2 + y_2^2}
    =\cos^2\theta-(\lambda\varphi-E)\sin^2\theta.
\end{align}
By the choice of $K$, $\theta'(x)
        \leq K$
uniformly in $x$.
Consequently, if $|x_2-x_1|\leq  \varepsilon \leq s/(2K)$,
\[
\arg F_{[x_1,x_2]}(w)
=\theta(x_2)
\leq \arg w + K\varepsilon 
< \arg w + s.\]
Moreover, \eqref{eq:argderivative} also implies that $
\theta'(x) > 0$ if $\theta(x)>0$ is small enough.
Putting these two observations together shows that any $\varepsilon > 0$ with $\varepsilon \leq s/(2K)$ satisfies~(a) and~(b).
\medskip

To address part~(c), first observe that
\begin{align} \label{eq:normderiv}
      \frac{\rmd r}{\rmd x}
      = \frac{1}{r}(y_1^{}y_1' + y_2^{}y_2')
      = r\cos\theta\sin\theta (\lambda \varphi - E + 1).
\end{align}
Thus, for any $x$ for which $\theta(x) \in [0,\pi/2]$, one has
\begin{equation} \label{eq:dLogRbound}
    \frac{\rmd}{\rmd x}\log r = \frac{1}{r} \frac{\rmd r}{ \rmd x} \geq -K.
\end{equation}

To deal with the case $\theta(x) \in (\pi/2,\pi/2+s)$, notice first that combining \eqref{eq:argderivative} and  \eqref{eq:normderiv} gives
$$
\frac{1}{\cos \theta\sin\theta}\frac{1}{r}\frac{\rmd r}{\rmd x}
=\lambda\varphi-E+1=\frac{1}{\sin^2\theta}-\frac{1}{\sin^2\theta}\frac{\rmd\theta}{\rmd x},
$$
and thus
$$
\frac{1}{r}\frac{\rmd r}{\rmd x}
= -\frac{\cos\theta}{\sin \theta} \frac{\rmd\theta}{\rmd x} + \frac{\cos\theta}{\sin\theta}.
$$
Denote the set of $x \in (x_1,x_2)$ with $\theta(x)>\pi/2$ by $\bigcup (a_i,b_i) \equiv \bigcup I_i$.
We have
\begin{align*} 
\int_{a_i}^{b_i}\frac{1}{r}\frac{\rmd r}{\rmd x} \rmd x
& = -\int_{a_i}^{b_i}\frac{\cos\theta}{\sin \theta}\frac{\rmd\theta}{\rmd x} \, \rmd x  
+\int_{a_i}^{b_i} \frac{\cos\theta}{\sin\theta} \, \rmd x
 \\
& \geq -\int_{\theta(a_i)}^{\theta({b_i})}\frac{\cos\theta}{\sin\theta} \, \rmd\theta- |I_i|\left|\cot\left(\frac{\pi}{2}+s\right)\right| .
\end{align*} 
For any $I_i$ with $\theta(a_i) =\theta(b_i)$, the first term drops out.
Otherwise, that first term is nonnegative, so one arrives at
\begin{align}  \label{eq:dLogRboundBad}
\int_{a_i}^{b_i}\frac{1}{r}\frac{\rmd r}{\rmd x} \rmd x
  \geq  - |I_i|\left|\cot\left(\frac{\pi}{2}+s\right)\right| 
  \geq -3s|I_i|
\end{align} 
for small enough $s>0$.

Combining \eqref{eq:dLogRbound} with \eqref{eq:dLogRboundBad} produces the desired result.
\end{claimproof}

\begin{claim}
 For any $\varepsilon>0$, there exists a partition of the interval $[a,b]$,
$$
t_0=a<t_1<t_2<\cdots <t_{N-1}<t_N=b,
$$
such that 

\begin{enumerate}[label={\rm(\alph*)}]
\item For every $i$, one has $|t_{i+1}-t_i|<\varepsilon$;

\item If $\varphi(t) = 0$ for some $t \in (t_{i}, t_{i+1})$, then $\varphi > 0$ throughout any interval of the partition that is adjacent to $[t_i, t_{i+1}]$. 
\end{enumerate}
\end{claim}

\begin{claimproof}
    Choose a partition $ a=s_0 < \cdots < s_M = b$ of $[a,b]$ into intervals of length at most $\varepsilon/2$. 
    Then, for each $j=1,2,\ldots,M$, the assumption that $\varphi^{-1}(\{0\})$ is nowhere dense permits us to choose a nondegenerate interval $[t_{2j-1}, t_{2j}] \subseteq (s_{j-1}, s_{j})$ on which $\varphi$ is positive, yielding the desired $\{t_i\}$ after putting $t_0=a$, $N = 2M+1$, and $t_N = b$.
\end{claimproof}

\vspace{3pt}

\begin{claim}
If an interval $[x_1, x_2]\subseteq [a,b]$ is such that  $\varphi > 0$ throughout $[x_1,x_2]$, then there exists $\lambda_0=\lambda_0(x_1, x_2, \varphi, K, s)$ such that for any $\lambda\ge \lambda_0$ and any $w\in V_s$, we have $F_{[x_1, x_2]}(w)\in V_{-s}$. 
Moreover, $|F_{[x_1, x_2]}(w)|\ge \frac{1}{2}|w|$ as $\lambda\to +\infty$.
\end{claim}

\begin{claimproof}
Fix $C = C(x_1,x_2,s)>0$ large enough that
\begin{equation} \label{eq:claim3C-choice}
\sin^2(2s)-C\cos^2(2s) < -4s(x_2-x_1)^{-1}.
\end{equation}
Since $\varphi$ is continuous, we may then choose
$\lambda_0>0$ large enough that $\lambda\varphi(x) - E \geq C$ throughout $[x_1,x_2]$ for all $\lambda \geq \lambda_0$.
For any $x \in [x_1,x_2]$ for which $|\theta(x)-\pi/2|<2s$, \eqref{eq:argderivative} then yields
\begin{equation} \label{eq:claim3argbound}
\frac{\rmd \theta}{\rmd x}
= \cos^2\theta -(\lambda\varphi-E)\sin^2\theta
\leq \sin^2(2s)-C\cos^2(2s)
<-4s(x_2-x_1)^{-1}
.    
\end{equation}
Using \eqref{eq:claim3argbound} shows that for any $w \in V_s \setminus V_{-s}$, one has $\arg F_{[x_1,x_2]}(w) < \tfrac{\pi}{2}-s$.
Recalling that $\theta'>0$ when $\theta>0$ is small, this proves the first part of the claim.

Let us now justify the second part of the claim. Suppose $w\in V_s\backslash V_{-s}$ (otherwise there is nothing to prove, since $r(x)$ is increasing if $\theta(x)\in (0,\pi/2)$ and $\lambda$ is large; see (\ref{eq:normderiv})). From (\ref{eq:argderivative})  and (\ref{eq:normderiv}) we have, 
for $\theta\in \left[\frac{\pi}{2}-s, \frac{\pi}{2}+s\right]$ and $\lambda$ sufficiently large,
$$
\left|\frac{\rmd r}{\rmd\theta}\right|=\left|\frac{r\cos\theta\sin\theta (\lambda \varphi - E + 1)}{\cos^2\theta-(\lambda\varphi-E)\sin^2\theta}\right|\le
\frac{r{\left|\cos \theta\right|}}{\left|\frac{\cos^2\theta}{1+\lambda\varphi-E}-\frac{\lambda\varphi-E}{1+\lambda\varphi-E}\sin^2\theta\right|}\le 2r {\left|\cos \theta\right|}.
$$
Hence, if $w\in V_s\backslash V_{-s}$, then 
$$
\left|\frac{\rmd}{\rmd\theta}\log r\right|\le 2{\left|\cos\theta\right|}\le 2\left|\theta-\frac{\pi}{2}\right|\le 2s.
$$
 Set $r_0=|w|$ and $r_1=|F_{[x_1, x^*]}(w)|$, where $x^*\in [x_1, x_2]$ is such that $\theta(x^*)=\frac{\pi}{2}-s$. Then we have $\log r_1\ge \log r_0-4s^2$, and $r_1\ge e^{-4s^2}r_0\ge \frac{1}{2}r_0$, since $s$ is small. Since $r(x)$ is increasing for $\theta(x)\in \left(0, \frac{\pi}{2}\right)$, this implies that $|F_{[x_1, x_2]}(w)|\ge \frac{1}{2}|w|.$ 
\end{claimproof}
\medskip

\begin{claim}
If an interval $[x_1, x_2]\subseteq [a,b]$ is such that $\varphi(x)>0$ for all $x\in [x_1, x_2]$, then for any   $w\in V_0$, we have  $|F_{[x_1, x_2]}(w)|\to +\infty$ as $\lambda\to +\infty$ uniformly in $E\le K$.
\end{claim}

\begin{claimproof}
This claim follows from applying  Lemmas~\ref{lem:autCompare}, \ref{lem:normgrowth}, and \ref{lem:C2growth} to \eqref{eq:nonautFlow} with the initial conditions given by the vector $w$. 
\end{claimproof}
\medskip

Taken together, these claims imply that the monodromy matrix $\monodromy_{E, \lambda}$ can be represented as a product of matrices such that after an application of the first one or two matrices in this product,  the image of each of the unit coordinate vectors will be in the first quadrant, remain there afterwards, and their images can be made arbitrarily large by choosing a sufficiently large value of $\lambda$. This implies that $\tr \monodromy_{E, \lambda}\to \infty$ as $\lambda\to \infty$, uniformly in $E \leq K$. 
\end{proof}

With Lemma~\ref{lem:traceAsymptoticsLargelambda}, we are now able to prove Theorem~\ref{t:main:nwd}.

\begin{proof}[Proof of Theorem~\ref{t:main:nwd}]
Fix $S \subseteq \R$ compact, and denote $\mathbf{A}=\monodromy_0$ and $\mathbf{B} = \monodromy_1$. 
Notice that $\mathbf{A}$ is independent of $\lambda$, and if the value of the energy $E$ is restricted to a fixed compact set $S$, then $\mathbf{A}$ is uniformly bounded: there exists $M = M(S)  >1$ such that $\|\mathbf{A}(E)\|\le M$ for every $E\in S$.  
  Pick  $K \gg M>1$ sufficiently large. On account of Lemma~\ref{lem:traceAsymptoticsLargelambda}, we know that $\left|\tr(\mathbf{B}(E,\lambda))\right|>2K$ for all $E \in S$ when $\lambda$ is large enough. 
We now consider two cases.

\vspace{3pt}

\textbf{\boldmath Case 1: $\left|\tr(\mathbf{A}\mathbf{B})\right|>2K^{1/3}$.}
Then,  we have
$$
\left|x_2\right|=\left|\frac{1}{2}\tr(\mathbf{A}\mathbf{B})\right|>K^{1/3}, \quad
\left|x_1\right|=\left|\frac{1}{2}\tr(\mathbf{B})\right|>K, \quad
\left|x_0\right|=\left|\frac{1}{2}\tr(\mathbf{A})\right|\le M,
$$
with $K\gg M$. 
It follows that the corresponding energy $E$ has an unbounded trace orbit (provided that $K$ is large enough) by standard arguments (e.g., \cite[Theorem~10.5.4]{DF2024ESO2}) and hence is not in the spectrum on account of Proposition~\ref{p.sigmab}.
\medskip

\textbf{\boldmath Case 2: $\left|\tr(\mathbf{A}\mathbf{B})\right|\le 2K^{1/3}$.} 
Defining $\{x_k\}$ as before, we have
$$
\left|x_2\right|=\left|\frac{1}{2}\tr(\mathbf{A}\mathbf{B})\right|\le K^{1/3}, \quad \left|x_1\right|=\left|\frac{1}{2}\tr(\mathbf{B})\right|>K, \quad  \left|x_0\right|=\left|\frac{1}{2}\tr(\mathbf{A})\right|\le M,
$$
we can bound the value of the Fricke--Vogt invariant:
\begin{align*}
I =     x_0^2+x_1^2+x_2^2-2x_0x_1x_2-1
    & \gtrsim K^2.
\end{align*}
Therefore, if $E \in S$ belongs to  the spectrum, the local Hausdorff dimension at $E$ can be estimated from above via Theorem~\ref{t.fiblocaldim}, and the upper bound on the Hausdorff dimension tends to zero as $\lambda\to \infty$, since $K\to \infty$ as $\lambda \to \infty$.
\end{proof}

We now turn to the proof of Theorem~\ref{t:counterexample}. We will in fact show that $C^1$ functions that alternate from positive to negative supply a family of counterexamples.
\begin{definition}
    Let us say that $f \in C^1([0,1])$ is a \emph{split} function if $f(0)=f(1)=0$ and there exists $x_* \in (0,1)$ such that
    \begin{enumerate}[label={\rm(\arabic*)}]
        \item $f(x)>0$ for $0<x<x_*$;
        \item $f(x)<0$ for $x_*<x<1$;
        \item $f'(x_*)<0$ and $f'(1)>0$.
        \end{enumerate}
     Notice that these assumptions force $f(x_*) = 0$.
\end{definition}

\begin{theorem} \label{t:counterexampleContinuous}
Suppose $f_0 =0 \cdot \chi_{[0,1)}$ and $f_1 \in C^1([0,1])$ is a split function.
Then, there exist $E \in \R$, $\lambda_k \uparrow \infty$ such that  $E\in \Sigma_{\lambda_k} := \Sigma(\lambda_k f_0,\lambda_k f_1)$ and 
\begin{equation}
\dim_\Hd^\loc(\Sigma_{\lambda_k},E)
=1
\end{equation} 
for all $k$.
\end{theorem}

\begin{remark}
    We suspect that Theorem \ref{t:counterexampleContinuous} still holds if assumption (3) in the definition of the split function is removed, but most likely a different argument is required.
\end{remark}

Writing a solution to the nonautonomous flow \eqref{eq:nonautFlow} in polar coordinates as in the proof of Lemma~\ref{lem:traceAsymptoticsLargelambda}, and introducing $L=\log r$, \eqref{eq:argderivative} and \eqref{eq:normderiv} can be rewritten as
\begin{equation}\label{e.pcoord}
  \begin{cases}
    &\displaystyle{\frac{\rmd \theta}{\rmd x}}
    = 1- (\lambda\varphi-{E}+1)\sin^2\theta,\\[5mm]
    &\displaystyle{\frac{\rmd L}{\rmd x}}
    = \cos \theta\sin\theta (\lambda\varphi - E + 1).
  \end{cases}
\end{equation}

Let us consider the solutions of this system separately on the intervals $(0,x_*)$ and $(x_*,1)$.
The arguments that follow involve various constants that may depend on $f$ and $E$, but which must be independent of $\lambda$. 
We write $f \lesssim g$ if $f \leq Cg$ for such a constant $C$, and $f \asymp g$ if both $f \lesssim g$ and $g \lesssim f$. Here we emphasize that the (eventual) applications of Lemmata~\ref{lem:B1big}--\ref{l.Mepsminusone} below will be for a fixed energy, so having energy-dependent constants is not problematic for the applications that we have in mind.

\begin{lemma} \label{lem:B1big}
Assume that $\varphi$ is continuous on $[a,b]$ and strictly positive on $(a,b)$, and fix $E \in \R$. Then $L(b) - L(a) \gtrsim \sqrt{\lambda}$, where $L(x)$ is any solution of the system \eqref{e.pcoord} with   $\theta(a) \in (0,\pi/2)$. 
\end{lemma}

\begin{proof}
Without loss of generality, consider $L(a)=0$.
The assumption  $\theta(a) \in (0,\pi/2)$ gives $y(a), y'(a) >0$.
Let us consider first the case in which $\varphi$ is strictly positive on $[a,b]$.
By assumption we can choose $C_2 > C_1 > 0$ such that for all $\lambda$ large enough, we have for all $x\in[a,b]$,
$$
C_1 \lambda < \lambda\varphi(x)-E+1 < C_2 \lambda.
$$
Note that for such $\lambda$, any zero of the right hand side of the first equation of \eqref{e.pcoord}, $\theta^*(x)$, must obey
\begin{equation} \label{eq:theta*}
\frac{1}{\sin^2 \theta^*(x)} = \lambda\varphi(x)-E+1 \in (C_1 \lambda, C_2 \lambda).
\end{equation}
We choose $\theta^*(a) \in (0,\pi/2)$ and then continuously choose $\theta^*$ on $(a,b)$.
In view of \eqref{eq:theta*}, we note that there are constants ${C}_2' >  {C}_1' >0$ such that for all $x$ and $\lambda$ under consideration:
\begin{equation}
    \theta^*(x) \in \left( \frac{{C}_1'}{\sqrt{\lambda}} ,
    \frac{{C}_2'}{\sqrt{\lambda}}\right).
\end{equation}
We also observe that $\theta'(x) > 0$ (resp., $ < 0$) if $\sin^2 \theta(x) < \sin^2 \theta^*(x)$ (resp., $\sin^2 \theta(x) > \sin^2 \theta^*(x)$.) 
In particular, $\theta(x)$ remains in $(0,\pi/2)$ for all $x$, so the result follows directly by applying Lemma~\ref{lem:C2growth} to the function $R=y$ (and noting that $L \geq \log y$).

Now consider the case in which $\varphi$ vanishes at one or more endpoints  by noting first that \eqref{e.pcoord} implies that $\theta'=1$ when $\theta=0$ and $\theta'\leq E+1$ everywhere, so since $\theta(a) \in (0,\pi/2)$, we may fix $\varepsilon>0$ small enough to ensure that $\theta(a+\varepsilon) \in (0,\pi/2)$ (the smallness condition depends on $E$, which is fixed for this argument).
We may then apply the argument from the previous case on $[a+\varepsilon,b-\varepsilon]$   to obtain
\[ L(b-\varepsilon) - L(a+\varepsilon) \gtrsim \sqrt{\lambda}. \]
On the boundary intervals, we have $L' \geq -E$, so the total change in $L$ is
\begin{equation}
    L(b) - L(a) \gtrsim \sqrt{\lambda} - 2E\varepsilon \gtrsim \sqrt{\lambda},
\end{equation}
where we permit an $E$-dependent constant in the final estimate.
\end{proof}

Let us now study what happens on an interval on which  $\varphi \leq 0$. 

\begin{lemma}\label{l.rangeofrotation}
Assume that $\varphi$ is continuous and strictly negative on $(c,d)$ and that $E > 0$. If $\theta$ denotes a solution of \eqref{e.pcoord}, 
then
\begin{equation}\label{e.largeangle}
\theta(d) - \theta(c) \asymp \sqrt{\lambda} \; \text{ for large } \lambda.
\end{equation}
\end{lemma}

\begin{proof}
Consider first the case when $\varphi$ is strictly negative throughout $[c,d]$.
Since $\varphi$ is strictly negative and bounded away from zero, one can choose $C>0$ such that  $\lambda\varphi(x)-{E}+1 \le - C \lambda$ for all sufficiently large $\lambda$.
In view of \eqref{e.pcoord}, one has
\begin{equation}
    \frac{\rmd \theta}{\rmd x} = 1-(\lambda\varphi-E+1)\sin^2\theta\geq 
    1+ C\lambda\sin^2\theta.
\end{equation}
In particular, the map $x \mapsto \theta(x)$ can be inverted. 
Thus, viewing $x$ as a function of $\theta$ and changing variables, we have
\begin{align*}
d-c 
 = \int_{c}^{d}\rmd x 
& = \int_{\theta(c)}^{\theta(d)}\frac{\rmd \theta}{1 -(\lambda\varphi(x(\theta))-{E}+1) \sin^2\theta } \\
& \le \int_{\theta(c)}^{\theta(d)}\frac{\rmd\theta}{1+ C \lambda \sin^2\theta} \\
& = \frac{1}{C\lambda} \int_{\theta(c)}^{\theta(d)}\frac{\rmd \theta}{\sin^2\theta + \frac{1}{C\lambda}}.
\end{align*}
Since $\int_0^\pi \frac{\rmd \theta}{\sin^2\theta+a^2}=\frac{\pi}{a\sqrt{1+a^2}}$, using $\pi$-periodicity of the integrand, we obtain the lower bound in \eqref{e.largeangle}. A completely analogous proof establishes the upper bound in \eqref{e.largeangle}, noting that the boundedness of $\varphi$ allows us to estimate $\lambda\varphi(x)-{E}+1 \ge - \widetilde C \lambda$ for $\widetilde C > 0$ and $\lambda$ sufficiently large.

Now consider the case in which $\varphi$ vanishes at one or more endpoints. Fixing $\varepsilon>0$, we may apply the previous argument on $[c+\varepsilon,d-\varepsilon] \subseteq (c, d)$  to obtain
\begin{equation}
    \theta(d-\varepsilon) - \theta(c+\varepsilon) \asymp \sqrt{\lambda}.
\end{equation}
On the boundary intervals, we notice that we still have $\rmd\theta/\rmd x \geq E \sin^2\theta  \geq 0$ and the one-sided bound $\lambda \varphi - E + 1 \geq - \widetilde C\lambda$ for a suitable $\widetilde  C>0$, which enables us to reprise the previous argument to deduce
\[0 \leq \theta(c+\varepsilon)-\theta(c) \lesssim \sqrt{\lambda}\]
and a similar statement on the other boundary interval, which suffices to derive the desired bounds.
\end{proof}

{The final task is to estimate the asymptotics of $L$ on $[x_*,1]$, the interval of non-positivity.} This is the most delicate part of the analysis, which we break into yet smaller steps. Without loss of generality, we can assume $\min \varphi = -1$. 
We will then break $[x_*,1]$ into three regimes: for well-chosen $0< \varepsilon_1 < \varepsilon_2 < 1$, we will consider the \emph{microscopic} ($-\lambda^{\varepsilon_1} \leq \lambda \varphi \leq 0$), \emph{mesoscopic} ($-\lambda^{\varepsilon_2} \leq \lambda \varphi \leq -\lambda^{\varepsilon_1}$), and \emph{macroscopic} ($-\lambda  \leq \lambda\varphi \leq - \lambda^{-\varepsilon_2}$) regions.
One further technical point merits attention here: the intervals defining the various regimes are themselves dependent on $\lambda$, so to appropriately track the $\lambda$-dependence, we must incorporate the length of the intervals into some of the estimates below.

\begin{lemma} \label{lem:B2small}
Assume that $\varphi$ is continuously differentiable and strictly negative on $[c,d]$ and that $E>0$.
If $L(x)$ is a solution of \eqref{e.pcoord} then $|L(d) - L(c)| \lesssim \log \lambda$ 
 as $\lambda\to \infty$.
\end{lemma}

\begin{proof}
Notice that \eqref{e.pcoord} implies  
\begin{equation}\label{e.dLdtheta}
\frac{\rmd L}{\rmd \theta}
=\frac{\cos \theta \sin\theta}{-\frac{1}{E-\lambda\varphi(x(\theta))-1} - \sin^2\theta},
\end{equation}
where we use monotonicity to view $x$ as a function of $\theta$ as in the proof of Lemma~\ref{l.rangeofrotation}; thus, our goal is to change variables and use
\begin{equation}
    L(d) - L(c)
    = \int_{\theta(c)}^{\theta(d)} \frac{\rmd L}{\rmd \theta} \, \rmd \theta.
\end{equation}
As before, we can assume that $\lambda$ is large enough that $\lambda \varphi(x) - E +1 \leq - C\lambda$ for all $x$ and $\lambda$ under consideration.
As in the proof of Lemma~\ref{l.rangeofrotation}, using
\begin{equation} \label{eq:dxdtheta}
\frac{\rmd x}{\rmd \theta}
= \frac{1}{1 - (\lambda\varphi - E + 1)\sin^2\theta}
\leq
\frac{1}{1+ C\lambda \sin^2\theta},
\end{equation}
we see that on any $\theta$ interval of the form $[\pi n, \pi(n+1)]$, $x$ can change by no more than $\lesssim \lambda^{-1/2}$.

Under the given assumptions, $\varphi(x)$ is Lipschitz continuous. 
Due to Lemma~\ref{l.rangeofrotation},  $\theta$ is changing over some interval $[\theta(c), \theta(d)]$, where $\theta(d) - \theta(c) \asymp \sqrt{\lambda}$. 
Let us split that interval into sub-intervals of length $\pi$. 
Notice that if $\varphi$ were constant, the integral of the right hand side of \eqref{e.dLdtheta}  over $[\pi n, \pi(n+1)]$ would be zero (since the denominator is symmetric with respect to the center of the interval $[\pi n, \pi(n+1)]$, and the numerator is anti-symmetric). 

By Lipschitz continuity of $\varphi$ we have
$$
\max_{[\pi n, \pi(n+1)]}\varphi(x(\theta))-\min_{[\pi n, \pi(n+1)]}\varphi(x(\theta))\lesssim  \frac{1}{ \sqrt{\lambda} }.
$$
If $|\varphi(x(\theta)) - \varphi(x(\theta_0))| \lesssim \lambda^{-1/2}$, then
$
-\frac{1}{E-\lambda\varphi(x(\theta))-1}
$ 
is equal to $
-\frac{1}{E-\lambda\varphi(x(\theta_0))-1}
$ up to a correction of order at most $\lambda^{-3/2}$. Therefore, since 
$$
\int_{\pi n}^{\pi (n+1)}\frac{\cos \theta \sin\theta}{-\frac{1}{E-\lambda\varphi(x(\theta_0))-1} -\sin^2\theta}\, \rmd\theta=0,
$$
we claim that the integral 
$$
\int_{\pi n}^{\pi (n+1)}\frac{\cos \theta \sin\theta}{ -\frac{1}{E-\lambda\varphi(x(\theta))-1}  -\sin^2\theta} \, \rmd\theta
$$
has order at most $\lambda^{-1/2}$. Indeed, we have
\begin{align}\nonumber
&\left|\int_{\pi n}^{\pi (n+1)}\frac{\cos \theta \sin\theta}{ -\frac{1}{E-\lambda\varphi(x(\theta))-1}  -\sin^2\theta} \, \rmd\theta\right|\\
\nonumber
& \qquad = \left|\int_{\pi n}^{\pi (n+1)}\frac{\cos \theta \sin\theta}{ -\frac{1}{E-\lambda\varphi(x(\theta))-1}  -\sin^2\theta} \, \rmd\theta-\int_{\pi n}^{\pi (n+1)}\frac{\cos \theta \sin\theta}{ -\frac{1}{E-\lambda\varphi(x(\theta_0))-1}  -\sin^2\theta} \, \rmd\theta\right| \\
\nonumber
& \qquad \lesssim \frac{1}{\lambda^{3/2}} \int_{\pi n}^{\pi (n+1)}\frac{|\cos \theta \sin\theta|}{(a^2 +\sin^2\theta)^2} \, \rmd\theta,
\end{align}
where $a^2=\min_{\theta\in[\pi n, \pi(n+1)]} \frac{1}{E-\lambda\varphi(x(\theta))-1}\asymp \frac{1}{\lambda}$. 

Using
$$
\int_0^{\pi/2}\frac{\cos \theta \sin \theta}{(a^2+\sin^2 \theta)^2}\rmd \theta
=\frac{1}{2a^2(a^2+1)},
$$
we get 
\begin{equation} \label{eq:trigIntegralBoundCombined}
\left|\int_{\pi n}^{\pi (n+1)}\frac{\cos \theta \sin\theta}{ -\frac{1}{E-\lambda\varphi(x(\theta))-1}  -\sin^2\theta} \, \rmd\theta\right| \lesssim \frac{1}{\sqrt{\lambda}}.
\end{equation}

Since $\theta(d)-\theta(c) \asymp \sqrt{\lambda}$, in the case in which the interval $[\theta(c), \theta(d)]$ splits into a union of intervals of the form $[\pi n, \pi(n+1)]$, the integral 
$$
\int_{\theta(c)}^{\theta(d)}\frac{\cos \theta \sin\theta}{- \frac{1}{E-\lambda\varphi(x(\theta))-1} -\sin^2\theta} \, \rmd\theta
$$
is at most of order $\sqrt{\lambda}\cdot \frac{1}{\sqrt{\lambda}}= 1$. 

At last, it remains to bound the integral $\int_J\frac{\cos \theta \sin\theta}{ -\frac{1}{E-\lambda\varphi(x(\theta))-1}  -\sin^2\theta} \, \rmd\theta$ over an interval $J$ of length smaller than $\pi$. 
Since the numerator is antisymmetric about points of the form $\pi(n+\frac{1}{2})$ and the denominator does not change sign, the integral in question is bounded from above by the absolute value of an  integral of the form 
$$
\int_{J'}\frac{\cos \theta \sin\theta}{ -\frac{1}{E-\lambda\varphi(x(\theta))-1}  -\sin^2\theta} \, \rmd\theta,
$$
where   $J'$ is of the form $[\pi n, \pi(n+1/2)]$ or $[\pi(n+1/2), \pi (n+1)]$ for an integer $n$.
Since $\int_0^{\pi/2}\frac{\cos(x)\sin(x)}{a^2+\sin^2(x)}\rmd x=\frac{1}{2}\ln\left(1+\frac{1}{a^2}\right)$, this implies that $\left|\int_{J'}\frac{\cos \theta \sin\theta}{ -\frac{1}{E-\lambda\varphi(x(\theta))-1}  -\sin^2\theta} \, \rmd\theta\right|\lesssim \log \lambda,$ which completes the proof of the lemma.
\end{proof}

Let us now treat the case when the function $\varphi$ is negative in the interior of an interval, but vanishes at the boundary points. 

\begin{lemma} \label{lem:zeroesonthe boundary}
Assume that $\varphi$ is continuously differentiable on $[c,d]$,  strictly negative on $(c,d)$, equal to zero at the end points $c$ and $d$,  $\varphi'(c)<0$ and $\varphi'(d)>0$, and that $E>0$.
If $L(x)$ is a solution of \eqref{e.pcoord}, then for large enough values of $\lambda$ we have $$|L(d) - L(c)| \lesssim \lambda^{\frac{9}{20}}.
$$
\end{lemma}

We will need a few statements  to prove Lemma \ref{lem:zeroesonthe boundary}.

\begin{lemma}\label{l.initial}
Suppose that on the interval $[c, d]$ for some $\varepsilon>0$  we have 
$$
-\lambda^{\varepsilon-1}\le \varphi\le 0,
$$
and let $L(x)$ be a solution of \eqref{e.pcoord}. Then for large $\lambda$ we have
$$
|L(d) - L(c)|
\lesssim \lambda^{\varepsilon}(d-c).
$$
\end{lemma}
\begin{proof}
    Without loss, consider $L(c)=0$. For $\lambda$ is large, one has 
    \[
    |\lambda\varphi-E+1|\lesssim C\lambda^\varepsilon.
    \]
    Therefore, the solution of the equation $L'
    = \cos \theta\sin\theta (\lambda\varphi - E + 1)$  cannot grow faster than the solution of the equation $\widetilde{L}'
    = C\lambda^\varepsilon.$
\end{proof}

\begin{lemma}\label{l.Mepsminusone}
Suppose that on the interval $[c, d]$ we have 
$$
-\lambda^{\varepsilon_2-1} \le \varphi\le -\lambda^{\varepsilon_1-1}
$$
for some $0< \varepsilon_1 < \varepsilon_2 \leq 1$, and let $L(x)$ be a solution of \eqref{e.pcoord} with the initial condition $L(c)=0$. 
Then for large $\lambda$ we have
$$
|L(d)|\lesssim (d-c)\lambda^{1+ \frac{3}{2}\varepsilon_2 - \frac{5}{2} \varepsilon_1}+\log \lambda.
$$ 
\end{lemma}
\begin{proof}
 Let $J_{\pi}$ denote an interval of length $\pi$. We have \begin{equation} \label{eq:repeat}
\frac{\rmd x}{\rmd \theta}
= \frac{1}{1 - (\lambda\varphi - E + 1)\sin^2\theta}
\leq
\frac{1}{C\lambda^{\varepsilon_1} \sin^2\theta+1},
\end{equation}
and since $\int_{J_{\pi}} \frac{\rmd \theta}{\sin^2\theta+a^2} =\frac{\pi}{a\sqrt{1+a^2}}$, on the interval $J_{\pi}$ the value of $x=x(\theta)$ changes by at most 
\begin{equation}
\int_{J_\pi} \frac{\rmd x}{\rmd \theta}\, \rmd \theta
    \lesssim  \lambda^{-\frac{\varepsilon_1}{2}}.
\end{equation}
Together with  Lipschitz continuity of $\varphi$, this implies that 
\[
\max_{x \in J_\pi} \varphi(x(\theta)) - \min_{x \in J_\pi} \varphi(x(\theta))
\lesssim \lambda^{-\frac{\varepsilon_1}{2}}.
\]
Hence, for any $\theta, {\theta_0}\in J_\pi$ we have
$$
\left|\frac{1}{E-\lambda\varphi(x(\theta))-1}-\frac{1}{E-\lambda\varphi(x({\theta_0}))-1}\right|
\lesssim\frac{\lambda\cdot \lambda^{-\varepsilon_1/2}}{\lambda^{2\varepsilon_1}}
= \lambda^{1-\frac{5\varepsilon_1}{2}}.
$$
Therefore, repeating the argument from the proof of Lemma~\ref{lem:B2small}, we get, for any $\theta_0\in J_\pi$, that
\begin{align*}&
\left|\int_{J_\pi}\frac{\cos \theta \sin\theta}{ -\frac{1}{E-\lambda\varphi(x(\theta))-1}  -\sin^2\theta} \, \rmd\theta\right|\\
&\qquad = \left|\int_{J_\pi}\frac{\cos \theta \sin\theta}{ -\frac{1}{E-\lambda\varphi(x(\theta))-1}  -\sin^2\theta} \, \rmd\theta-\int_{J_\pi}\frac{\cos \theta \sin\theta}{ -\frac{1}{E-\lambda\varphi(x(\theta_0))-1}  -\sin^2\theta} \, \rmd\theta\right| \\
& \qquad \lesssim \lambda^{1-\frac{5\varepsilon_1}{2}}\int_{J_\pi}\frac{|\cos \theta \sin\theta|}{(a^2 +\sin^2\theta)^2} \, \rmd\theta,
\end{align*}
where $a^2={\displaystyle{\min_{\theta\in J_\pi}}} \frac{1}{E-\lambda\varphi(x(\theta))-1}\ge \frac{1}{E+\lambda^{\varepsilon_2}-1}\gtrsim  \frac{1}{ \lambda^{\varepsilon_2}}$. Since $\int_0^{\pi/2}\frac{\cos x\sin x}{(a^2+\sin^2 x)^2} \, \rmd x =\frac{1}{2a^2(a^2+1)},$ we have 
$$
\left|\int_{J_\pi}\frac{\cos \theta \sin\theta}{ -\frac{1}{E-\lambda\varphi(x(\theta))-1}  -\sin^2\theta} \, \rmd\theta\right|\lesssim \lambda^{1-\frac{5\varepsilon_1}{2}} \cdot \lambda^{\varepsilon_2} =\lambda^{1+\varepsilon_2 -\frac{5\varepsilon_1}{2}}.
$$
Now, we claim that 
$$
\theta(d)-\theta(c)\lesssim \max(1, (d-c)\lambda^{\varepsilon_2/2}).
$$
Indeed, for large $\lambda$ we have
$$
\frac{\rmd \theta}{\rmd x} = 1-(\lambda\varphi-E+1)\sin^2\theta\le 1+{2\lambda^{\varepsilon_2}}\sin^2\theta, 
$$
hence
\begin{align*}
d-c
=\int_{c}^{d}\rmd x
& =\int_{\theta(c)}^{\theta(d)}\frac{\rmd\theta}{1-(\lambda\varphi(x(\theta))-E+1)\sin^2\theta}
\\
& \geq \int_{\theta(c)}^{\theta(d)}\frac{\rmd\theta}{1+2{\lambda^{\varepsilon_2}}\sin^2\theta} \\
& =
\frac{1}{2{\lambda^{\varepsilon_2}}}\int_{\theta(c)}^{\theta(d)}\frac{\rmd\theta}{\frac{1}{2{\lambda^{\varepsilon_2}}}+\sin^2\theta}.
\end{align*}
If $\theta(d)-\theta(c)> \pi$, then, taking into account that $\int_0^\pi \frac{\rmd \theta}{\sin^2\theta+a^2}
=\frac{\pi}{a\sqrt{1+a^2}}$, we have
\begin{align*}
   d-c
   \ge  \frac{1}{2{\lambda^{\varepsilon_2}}}\int_{\theta(c)}^{\theta(d)} \frac{\rmd\theta}{\frac{1}{2{\lambda^{\varepsilon_2}}} + \sin^2\theta}
   & \geq \frac{1}{2{\lambda^{\varepsilon_2}}}\left(\frac{\theta(d)-\theta(c)}{\pi}-1\right)\int_0^\pi\frac{\rmd\theta}{\frac{1}{2{\lambda^{\varepsilon_2}}}+\sin^2\theta}\\
   & \gtrsim \lambda^{-\varepsilon_2/2} \left(\frac{\theta(d)-\theta(c)}{\pi} -1\right),
\end{align*}
so $\theta(d)-\theta(c)\lesssim (d-c)\lambda^{\varepsilon_2/2}+1$.
Incorporating also the case in which $\theta(d)-\theta(c)\le \pi$, we get
$$
\theta(d)-\theta(c)
\lesssim \max(1, (d-c)\lambda^{\varepsilon_2/2}).
$$
Therefore, if the interval $[\theta(c), \theta(d)]$ splits into intervals of length $\pi$ exactly, we have
$$
|L(d)|\lesssim (d-c)\lambda^{\varepsilon_2/2} \lambda^{1+\varepsilon_2 - \frac{5\varepsilon_1}{2}}
=(d - c)\lambda^{1 + \frac{3}{2}\varepsilon_2 - \frac{5}{2} \varepsilon_1}. 
$$
Finally, if the interval $[\theta(c), \theta(d)]$   does not split into intervals of length $\pi$ exactly, we can get a remainder of order $\log \lambda$ in this estimate, as was shown in the proof of Lemma~\ref{lem:B2small}. This completes the proof of the lemma.
\end{proof}

Now are ready to prove Lemma \ref{lem:zeroesonthe boundary}. 

\begin{proof}[Proof of Lemma \ref{lem:zeroesonthe boundary}]

Fix $\varepsilon_1 = \frac{29}{40}$ and  $\varepsilon_2 = \frac{33}{40}$, let $\varphi$ denote a function satisfying the assumptions of the lemma; we may assume without loss that $\min \varphi = -1$. 
Since $\varphi'(c) < 0$ and $\varphi \in C^1$, there exists $\delta>0$ such that $\varphi'<0$ and hence $\varphi$ is decreasing on $[c,c+\delta]$.
It follows that for any sufficiently large $\lambda>0$, the interval $\left[c, \frac{c+d}{2}\right]$ can be split into the union 
$$
[c_1, d_1]\cup [c_2, d_2]\cup [c_3, d_3]
$$
of intervals with disjoint interior, so that $c=c_1$, $d_1=c_2$, $d_2=c_3,$ $d_3=\frac{c+d}{2},$ and the following hold:

\begin{enumerate}[label={\rm(\arabic*)}]

\item $-\lambda^{\varepsilon_1-1}\le \varphi\le 0$  if $x\in [c_1, d_1]$; 

\item $-\lambda^{\varepsilon_2-1} \le \varphi\le -\lambda^{\varepsilon_1-1}$ if $x\in [c_2, d_2]$; 

\item $-1 \le \varphi\le -\lambda^{\varepsilon_2-1}$ if $x\in [c_3, d_3]$.

\end{enumerate}

For large $\lambda>0$ we have 
$$
d_1-c_1\asymp \lambda^{\varepsilon_1-1},
\qquad d_2-c_2\asymp \lambda^{\varepsilon_2-1}, \qquad
d_3-c_3\asymp 1.
$$

Applying Lemma~\ref{l.initial} on $[c_1,d_1]$ and Lemma~\ref{l.Mepsminusone} twice (once as written on $[c_2,d_2]$ and a second time with $(\varepsilon_1,\varepsilon_2)$ replaced by $(\varepsilon_2,1)$ on $[c_3,d_3]$), we obtain 
$$
L(d_1)-L(c_1)\lesssim \lambda^{2\varepsilon_1-1},
\qquad L(d_2)-L(c_2)\lesssim \lambda^{\frac{5\varepsilon_2}{2}-\frac{5\varepsilon_1}{2}}, \qquad 
L(d_3)-L(c_3)\lesssim \lambda^\frac{5-5\varepsilon_2}{2}.
$$
In particular, for  our choice of the values of $\varepsilon_1, \varepsilon_2$ we have
$$
2\varepsilon_1-1= \frac{9}{20},
\qquad\frac{5\varepsilon_2}{2}-\frac{5\varepsilon_1}{2}=\frac{1}{4}<\frac{9}{20},
\qquad\frac{5-5\varepsilon_2}{2}=\frac{35}{80}<\frac{9}{20}. 
$$
Therefore, $L(d_3)-L(c_1)\lesssim \lambda^\frac{9}{20}$. 
Application of an analogous argument to the interval $\left[\frac{c+d}{2}, d\right]$ completes the proof.
\end{proof} 

\begin{proof}[Proof of Theorem~\ref{t:counterexampleContinuous}]
Fix $E$ of the form $E = 4\pi^2 n^2$, denote
$\mathbf{A} = \mathbf{M}_0$, $\mathbf{B} = \monodromy_1$, and define
\begin{equation}
    \mathbf{B}_1(E,\lambda)
    \eqdef \transferMatrix_{[0,x_*]}(\lambda f_1, E),
    \quad \mathbf{B}_2(E,\lambda)
    \eqdef \transferMatrix_{[x_*,1]}(\lambda f_1,E).
\end{equation}
Since $\mathbf{A} = \idty$, it suffices to show that we can choose $\lambda_k \to \infty$ such that $\mathbf{B} = \mathbf{B}(E,\lambda_k) = \mathbf{B}_2(E,\lambda_k)\mathbf{B}_1(E,\lambda_k)$ satisfies
$\mathbf{B}^2 = - \idty$.
Indeed, this would imply that $\mathbf{B}$ is elliptic and therefore  $E$ belongs to the spectrum.
Furthermore, since $\mathbf{A} = \idty$, it follows that $\mathbf{A}$ and $\mathbf{B}$ commute, and therefore $I(E)=0$, so that the local Hausdorff dimension of the spectrum at $E$ is one.

To that end, we note that $\log\|\mathbf{B}_1\| \gtrsim {\lambda}^{1/2}$ by Lemma~\ref{lem:B1big}.
By the proof of  Lemma~\ref{lem:B1big} and Proposition~\ref{prop:funfact}, the expanding direction of $\mathbf{B}_1$, $U(\mathbf{B}_1)$, satisfies 
\begin{equation}
\angle (\mathbf{B}_1U(\mathbf{B}_1), e_1)
\lesssim \exp(-\lambda^{1/2} ),
\end{equation}
where $e_1 = (1, 0)^\top$.  By Lemma~\ref{lem:zeroesonthe boundary}, $\log \|\mathbf{B}_2\| \lesssim \lambda^{9/20}$, so  $\log \|\mathbf{B}\| \gtrsim {\lambda}^{1/2} - \lambda^{9/20} \gtrsim \lambda^{1/2}$ with $\mathbf{B}_1 U(\mathbf{B}) \approx (1,0)^\top$ as well.
Then, by taking advantage of Lemma~\ref{l.rangeofrotation}, we can start at any large $\lambda$ value that we like and increase until we arrive at a $\lambda$ for which
\begin{equation}
    \mathbf{B}U(\mathbf{B}) = \mathbf{B}_2\mathbf{B}_1U(\mathbf{B}) = S(\mathbf{B}).
\end{equation}
Since $\mathbf{B}U(\mathbf{B}) = S(\mathbf{B}^{-1})$ is orthogonal to $U(\mathbf{B}^{-1}) = \mathbf{B}S(\mathbf{B})$, it follows that $\mathbf{B}$ exchanges the subspaces $\{U(\mathbf{B}), S(\mathbf{B})\}$, so $\tr\mathbf{B}=0$, which implies $\mathbf{B}^2=-\idty$ by the Cayley--Hamilton theorem, concluding the argument.
\end{proof}

\section{A Nonlinear Eigenvalue Formulation} \label{sec:nlevp}

\subsection{A Very Brief Review of Floquet Theory} \label{sec:floquet}

Suppose $V: \R \to \R$ is $p$-periodic with $p>0$ and
\begin{equation}
    \int_0^p \! |V(x)|^2 \, \rmd x < \infty.
\end{equation}
Let us recall a few notions from Floquet theory. For proofs, see \cite[Chapter~11]{LukicBook}; for additional background and context, see \cite{Kuchment2016BAMS}.
To study the spectral properties of $L_V = -\partial_x^2 + V$, consider the following family of boundary value problems, indexed by a parameter $\th \in [0,\pi]$:
\begin{equation} \label{eq:BVPtheta}
    -y''(x) + V(x)y(x) = Ey(x), \quad y(p)=e^{i\th}y(0), \ y'(p) = e^{i\th}y'(0).
\end{equation}
For each $E$ and $\th$, the solution space of \eqref{eq:BVPtheta} can be $0$-, $1$-, or $2$-dimensional.
For a given $V$, $p$, and $\th \in [0,\pi]$, write $E_1(\th;V,p) \leq E_2(\th;V,p) \leq \cdots$ for the spectrum of \eqref{eq:BVPtheta}, that is, the list of $E \in \R$ (with multiplicity) such that \eqref{eq:BVPtheta} has a nontrivial solution.

We have the following facts:
\begin{itemize}
    \item Each $E_j(\,\cdot\,;V,p)$ is monotonic on $[0,\pi]$. In fact, $E_j$ is \emph{increasing} if $j$ is odd and \emph{decreasing} if $j$ is even.
    \item The spectrum of $L_V$ is given by
    \begin{equation}
        \spectrum(L_V) = \{ E_j(\th;V,p) : \th \in [0,\pi], \ j \geq 1 \}.
    \end{equation}
\end{itemize} 
Combining these two points,
\begin{equation}
    \spectrum(L_V) = \bigcup_{j=1}^\infty \Big( [E_{2j-1}(0),E_{2j-1}(\pi)] \cup [E_{2j}(\pi) , E_{2j}(0)]\Big) .
\end{equation}
In particular, to identify the spectrum of $L_V$, it suffices to identify the points $E_j(\th)$ with $\th \in \{0,\pi\}$ and furthermore, one has
\begin{equation}
    \spectrum(L_V)
    = \{ E \in \R : \tr \transferMatrix_{[0,p]}(V,E) \in [-2,2]\}.
\end{equation}

\subsection{A Nonlinear Eigenvalue Problem for Periodic Approximations} 
We now return to the Fibonacci setting.
When $f_0$ and $f_1$ permit an exact solution formula for $-y''(x) + \lambda f_j(x)y(x) = E y(x)$, the Floquet characterization of the spectrum of periodic approximations to the continuum Fibonacci operator can be characterized by a parameterized nonlinear eigenvalue problem.
This approach resembles the framework proposed by Mennicken and M\"oller, e.g., with application to serially connected beams~\cite[Sect.~10.4]{MM03} and networks of strings~\cite{Bak23}, as well as the Wittrick--Williams method of dynamic stiffness matrices~\cite{WW71}.

We first consider the constant potentials $f_0 = 0\cdot \chi_{[0,1)}$ and $f_1 = \lambda \cdot \chi_{[0,1)}$ in the case of $k=2$, i.e., period $p=2$.  Let $y_1, y_2:[0,1]\to\C$ denote the solutions of~(\ref{eq:TimeIndepSchrod}) on the subdomains $[0,1]$ and $[1,2]$.\footnote{In the second case, we translate the argument so that $y_2$ is a function on $[0,1]$, not $[1,2]$.}
For the given potential, equation (\ref{eq:TimeIndepSchrod}) requires
\begin{equation} \label{eq:nlevp_ode}
-y_1''(x) + \lambda y_1(x) = E y_1(x), \qquad
   -y_2''(x) = Ey_2(x).
   \end{equation}
These constant-coefficient equations have the general solutions (for $E \ne \lambda$)
\begin{eqnarray*}
y_1(x) &=& \Acoef_1 \sin(\sqrt{E-\lambda}\, x) + \Bcoef_1 \cos(\sqrt{E-\lambda}\, x) \\[3pt]
y_2(x) &=& \Acoef_2 \sin(\sqrt{E}\,x) + \Bcoef_2 \cos(\sqrt{E}\,x).
\end{eqnarray*}
At the subdomain interface, continuity and smoothness ($y_1(1) = y_2(0)$, $y_1'(1) = y_2'(0)$)  require
\begin{eqnarray*}
\Acoef_1 \sin(\sqrt{E-\lambda}) + \Bcoef_1 \cos(\sqrt{E-\lambda}) &=& \Bcoef_2 \\[3pt]
\Acoef_1\sqrt{E-\lambda} \cos(\sqrt{E-\lambda}) - \Bcoef_1 \sqrt{E-\lambda} \sin(\sqrt{E-\lambda}) &=& \Acoef_2 \sqrt{E}.
\end{eqnarray*}
At the ends of the domain, the Floquet conditions~(\ref{eq:BVPtheta}) require $y_2(1) = e^{i\th} y_1(0)$ and $y_2'(1) = e^{i\th}y_1'(0)$, giving
\begin{eqnarray*}
\Acoef_2 \sin(\sqrt{E}) + \Bcoef_2 \cos(\sqrt{E}) &=& \Bcoef_1 e^{i\th}\\[3pt]
 \Acoef_2 \sqrt{E} \cos(\sqrt{E}) - \Bcoef_2 \sqrt{E} \sin(\sqrt{E}) &=& \Acoef_1 e^{i\th} \sqrt{E-\lambda}.
\end{eqnarray*}
Arrange these last four equations into the standard form 
{\footnotesize \[
\left[\!\!\begin{array}{cccc}
\sin(\sqrt{E-\lambda}) & \cos(\sqrt{E-\lambda}) & 0 & -1 \\[3pt]
\sqrt{E-\lambda}\, \cos(\sqrt{E-\lambda}) & - \sqrt{E-\lambda}\, \sin(\sqrt{E-\lambda}) & - \sqrt{E} & 0 \\[3pt]
0 & -e^{i\th} & \sin(\sqrt{E}) & \cos(\sqrt{E}) \\[3pt]
-e^{i\th} \sqrt{E-\lambda} & 0 & \sqrt{E} \cos(\sqrt{E}) & -\sqrt{E} \sin(\sqrt{E})
\end{array}\!\!\right]\!\!
\left[\!\!\begin{array}{c} \Acoef_1 \\[4pt] \Bcoef_1 \\[4pt] \Acoef_2 \\[4pt] \Bcoef_2 \end{array}\!\!\right]
= \left[\!\begin{array}{c} 0\\[4pt] 0\\[4pt] 0\\[4pt] 0 \end{array}\!\right]\!.
\]}
We seek values of $E$ for which the matrix on the left, $\mathbf{T}_{\lambda,\th}(E)$, 
is singular
(and the corresponding solution of the differential equation is nontrivial).
There will generally be infinitely many such real values of $E$.
A wide range of algorithms exist for finding these eigenvalues; see, e.g., G\"uttel and Tisseur~\cite{GT17}.  (Indeed, this family of examples could provide useful test problems for these algorithms.)

This period-2 case can be readily generalized to period~$p = F_k$, 
where $F_k$ is the $k$th Fibonacci number ($F_0=F_1=1$, $F_{k+1} = F_k + F_{k-1}$).
Let $\{\omega_n\}_{n\in\Z}$ denote the sequence of period $F_k$ 
determined from~(\ref{eq:fib}) but with $\alpha$ replaced by $\alpha_k:= F_k/F_{k+1}$  (and $\theta =0$).
Introduce
\[ \freq_n(E) = \left\{\begin{array}{cl}
                            \sqrt{E},         & \omega_n = 0; \\[3pt]
                            \sqrt{E-\lambda}, & \omega_n = 1.
                  \end{array}\right.
\]
Let $e_j\in\C^{2F_k}$ denote the $j$th column of the identity matrix.
Then we can express the nonlinear eigenvalue function
as $\mathbf{T}_{\lambda,\th}(E): \C^{2F_k} \to \C^{2F_k}$ as
\begin{eqnarray*}
\mathbf{T}_{\lambda,\th}(E) &=&
   \sum_{n=1}^{F_k}  \Big(\sin(\freq_n(E)) e_{2n-1}^{} e_{2n-1}^\top
                     + \cos(\freq_n(E)) e_{2n-1}^{} e_{2n}^\top \Big) \\[5pt]
   &&{}+ \sum_{n=1}^{F_k} \Big(\freq_n(E) \cos(\freq_n(E)) e_{2n}^{} e_{2n-1}^\top
                         -\freq_n(E) \sin(\freq_n(E)) e_{2n}^{} e_{2n}^\top \Big)  \\[5pt]
   &&{}- \sum_{n=1}^{F_k-1} e_{2n-1}^{} e_{2n+2}^\top
       - \sum_{n=1}^{F_k-1} \freq_{n+1}(E) e_{2n}^{} e_{2n+1}^\top  \\[5pt]
   &&{}- e^{i\th}  e_{2F_k-1}e_2^\top - \freq_1(E) e^{i \th} e_{2F_k} e_1^\top.
\end{eqnarray*}

\begin{figure}
\includegraphics[width=5.5in]{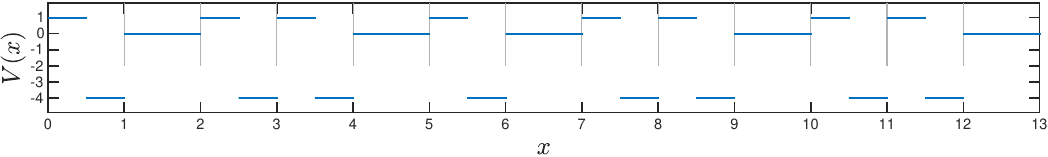}\quad

\medskip
\includegraphics[width=2.85in]{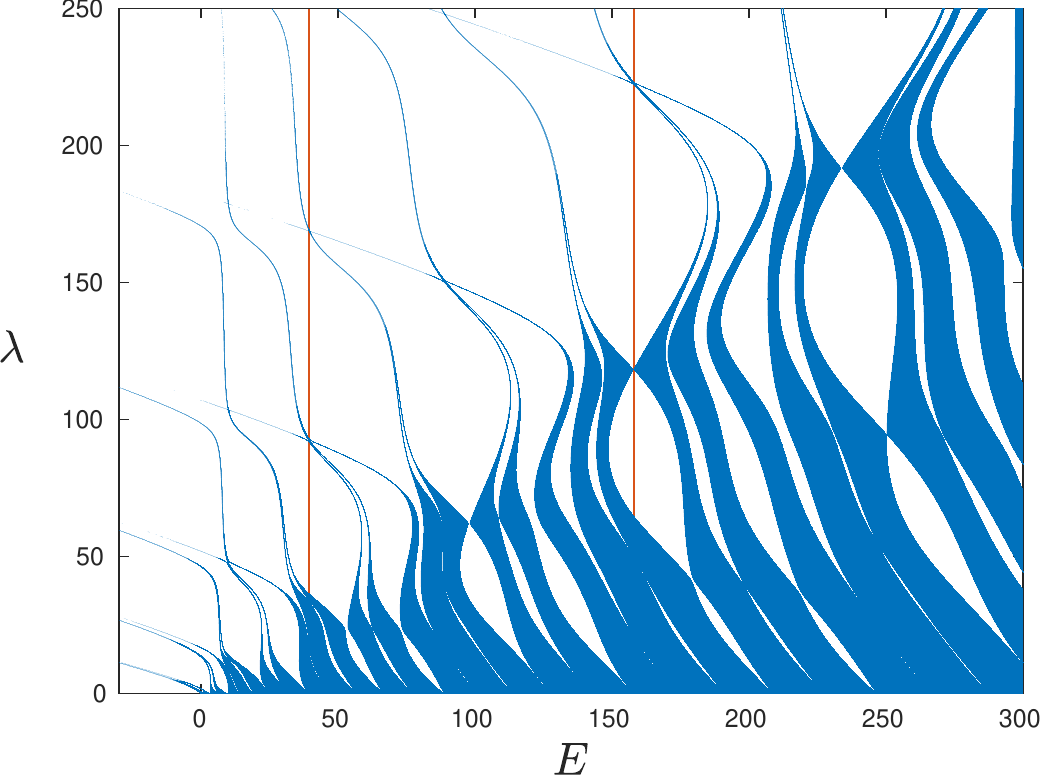}\quad
\includegraphics[width=2.85in]{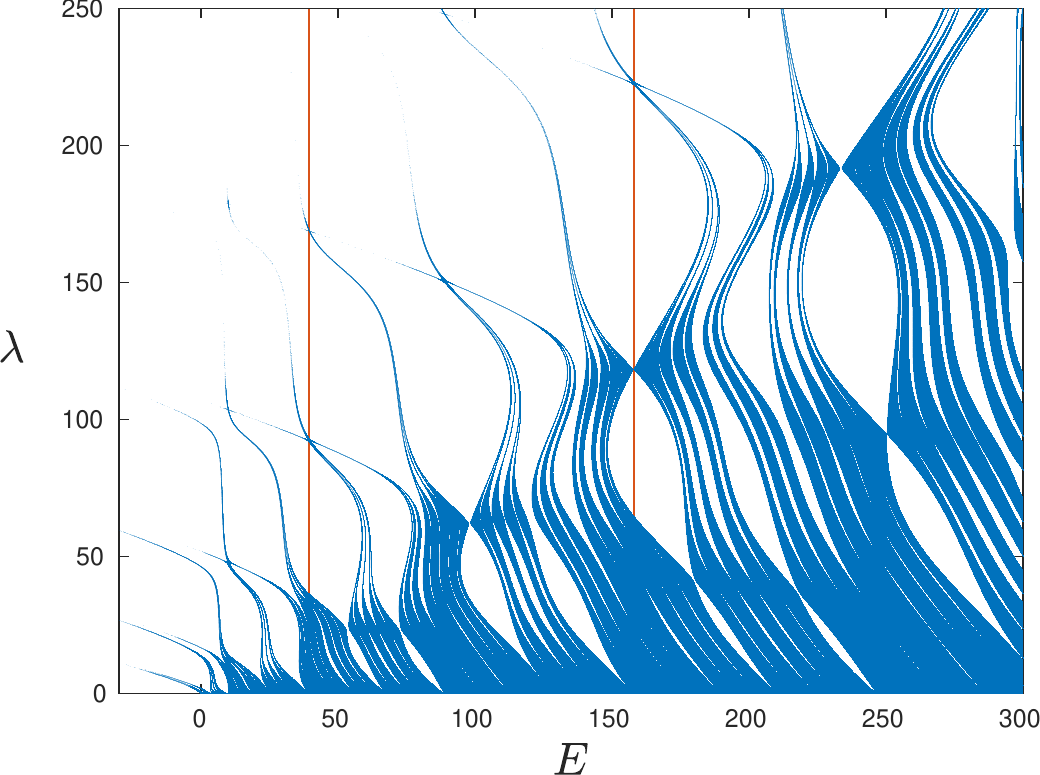}
\caption{\label{fig:prop32a}For the piecewise constant potential with $c=4$ (shown on the top), the bottom plots show a portion of the spectrum for periodic approximations ($p=F_4=5$, left; $p=F_6=13$, right).  The vertical red lines indicate the values of $E=4\pi^2$ and $E=16\pi^2$.}
\end{figure}

\begin{figure}
\includegraphics[width=2.55in]{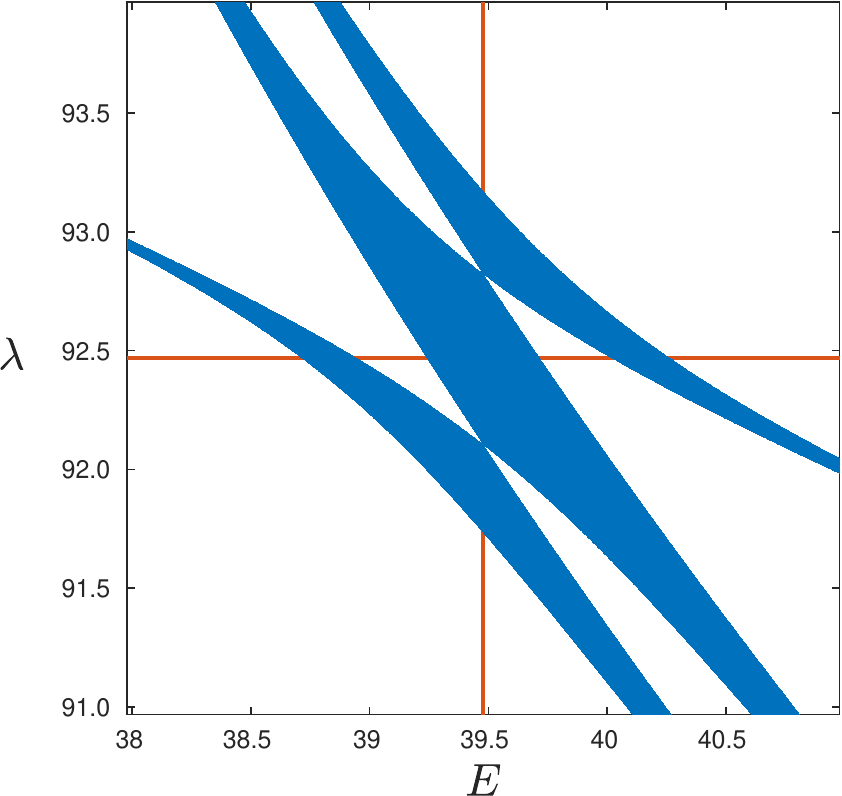}\quad
\includegraphics[width=2.55in]{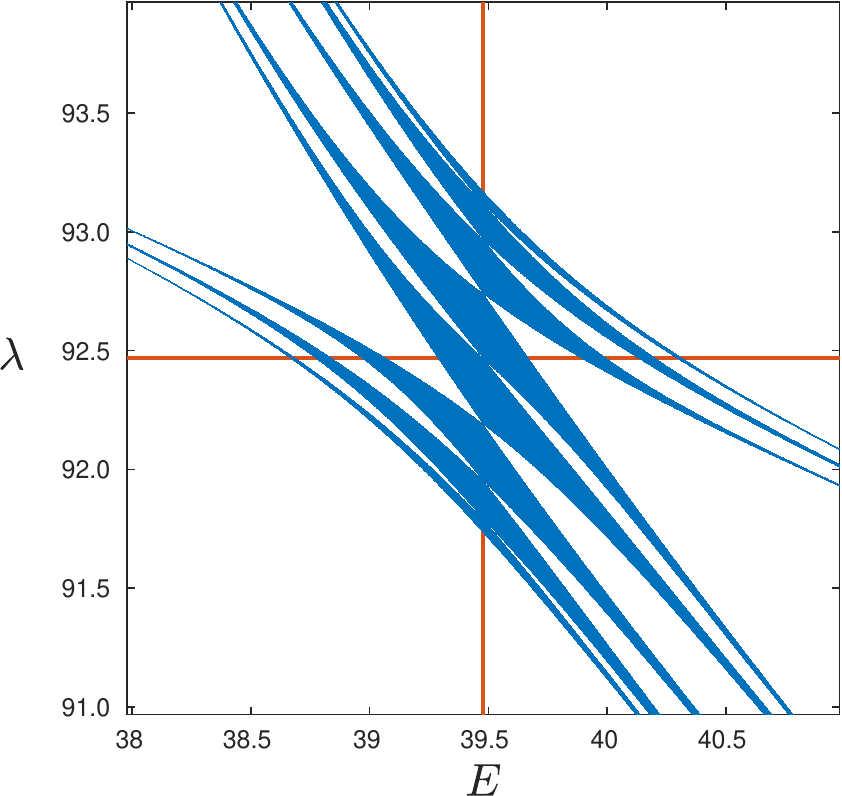}
\caption{\label{fig:prop32b} For the potential in Figure~\ref{fig:prop32a}, now focusing on the spectrum near $E=4\pi^2$ (vertical red line) and the first $\lambda>E$ for which ${\rm Tr}\ \mathbf{B}(E, \lambda)=0$ (horizontal red line), for periodic approximations ($p=F_4=5$, left; $p=F_6=13$, right).    We observe $F_3=3$ and $F_5=8$ bands in these plots.}
\end{figure}

Motivated by Theorem~\ref{t:counterexampleContinuous}, we consider $f_0 = 0 \cdot \chi_{[0,1)}$ and $f_1$  a simple discrete version of a split function:
\begin{equation}
    f_1(x) =  \begin{cases}
        1 & 0 \le x < 1/2 \\
        -c & 1/2 \le x < 1,
    \end{cases}
\end{equation}
where $c>0$ is a suitable constant.
Of course, $f_1$ is not $C^1$ and hence not a genuine split function, but we remark in passing that the arguments proving Theorem~\ref{t:counterexampleContinuous} can be adjusted to apply to this choice of $f_1$, which thus gives us a useful benchmark for the computations.
Since this potential is constant on each subdomain $[j/2, (j+1)/2]$ for $j\in\Z$, we can apply the approach just described to characterize the spectra of periodic approximations.  In this case, it is convenient (if slightly profligate) to solve the differential equation on half-unit subdomains, so that the order $p=F_k$ periodic approximation leads to a nonlinear eigenvalue problem of order $4F_k$ (rather than $2F_k$, as previously).  This formulation can be used to study the fine spectral structure for suitable potentials.  For example, Figure~\ref{fig:prop32a} shows a portion of the spectrum for periodic approximations of length $p=F_5= 8$ and $p=F_6=13$ for $c=4$. There exists an unbounded sequence $\{\lambda_k\}$ such that $E = 4\pi^2$ is in the spectrum $\Sigma_k$.  To investigate this phenomenon, Figure~\ref{fig:prop32b} focuses around $E=4\pi^2$ and the first value of $\lambda>E$ such that $\mathbf{B}(E, \lambda)=0$; the red lines show $\lambda=92.46773$ and $E = 4\pi^2$.  (These later plots resemble the spectra of periodic approximations to \emph{discrete} Fibonacci operators~\cite[Fig.~4]{DEG15}.)  In Figure~\ref{fig:thm3_7} we replace the piecewise constant $f_1$ by the \textit{bona fide} split function \[f_1(x) = 50 x (x-1/3) (x-1).\]  We observe similar spectral features, as expected from Theorem~\ref{t:counterexampleContinuous}.

For numerical investigations,
we complement this nonlinear eigenvalue formulation 
with a piecewise Chebyshev pseudospectral discretization~\cite{Tre00} 
of the differential equations.  This approach, which we used to create all the plots in this paper, has the advantage of allowing potentials that are not piecewise constant, as in the left plot in Figure~\ref{fig:f1pos} and in Figure~\ref{fig:thm3_7}.  The nonlinear eigenvalue formulation provides a way to benchmark the accuracy of the discretization method using examples with piecewise constant potentials.

\begin{figure}
\includegraphics[width=5.5in]{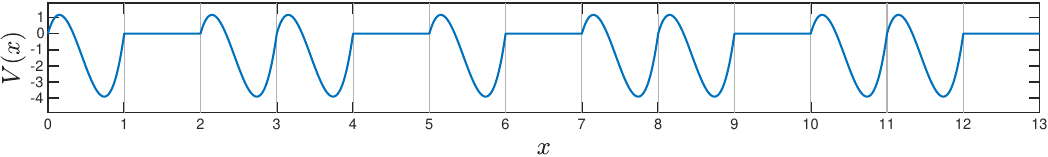}\quad

\medskip
\includegraphics[height=2.25in]{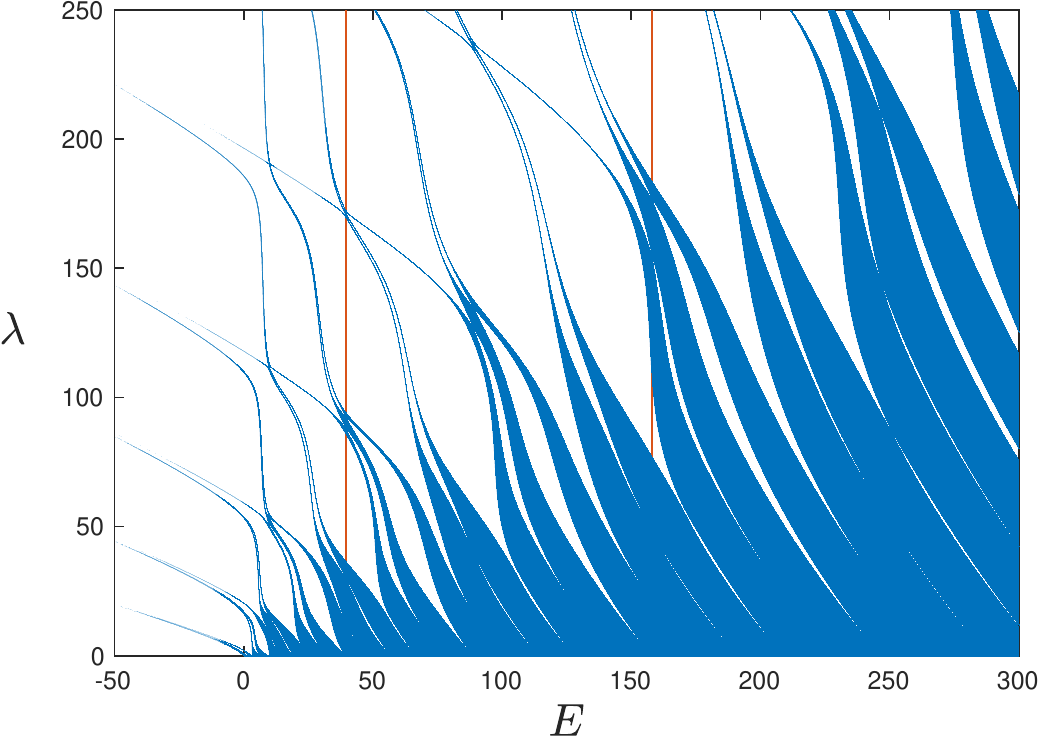}\quad
\includegraphics[height=2.25in]{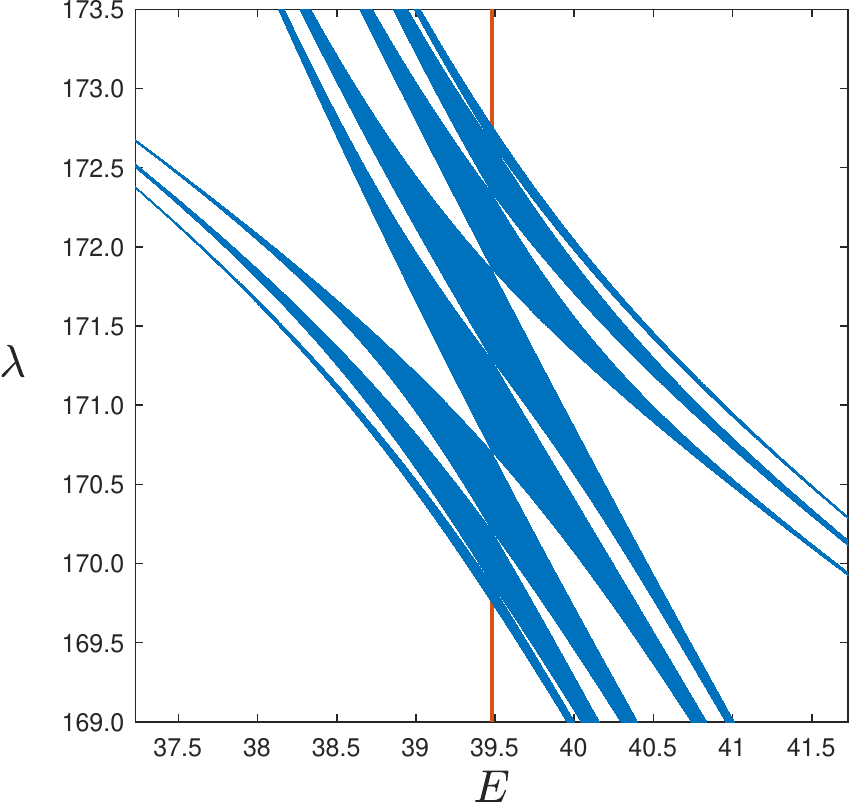}

\caption{\label{fig:thm3_7} For a smooth potential involving a split function $f_1$ (shown on the top), the bottom plots show a portion of the spectrum for the periodic approximation with $p=F_4=5$ (left), and a zoom for $p=F_6=13$ (right) showing a confluence of $F_5=8$ spectral bands around $E=4\pi^2$ (vertical red line).}
\end{figure}

\begin{appendix}
    \section{Singular Vectors of Hyperbolic Matrices} \label{sec:funfact}

We need the following elementary statement about expanding directions and expanding eigenspaces for hyperbolic $\SL(2,\R)$ matrices.
Here we recall that any $\mathbf{A} \in \SL(2,\R)$ with $\|\mathbf{A}\|>1$ has expanding and contracting directions $U(\mathbf{A}), S(\mathbf{A}) \in \R\mathbb{P}^1$ such that 
\begin{equation}
    \left|\mathbf{A}   u\right| = \|\mathbf{A}\|, \quad \left|\mathbf{A}   s \right| = \|\mathbf{A}\|^{-1}
\end{equation}
for any unit vectors $  u \in U(\mathbf{A})$, $  s \in S(\mathbf{A})$.

\begin{prop} \label{prop:funfact}
Suppose $\mathbf{A} \in \SL(2,\R)$ is hyperbolic with eigenvalues $\lambda^{\pm 1}$, where $|\lambda|>1$. 
Let $U(\mathbf{A})$ denote the direction that is most expanded by $\mathbf{A}$ and $V(\mathbf{A})$ the eigenspace with eigenvalue $\lambda$. Then,
\begin{equation}
    \angle(V(\mathbf{A}), \mathbf{A}U(\mathbf{A})) \leq \frac{\pi}{2  |\lambda| \|\mathbf{A}\|}.
\end{equation}
\end{prop}

\begin{proof}
For any unit vector $  v \in V(\mathbf{A})$, we have
\begin{equation}
    \left|\mathbf{A}^{-1}   v\right| = |\lambda|^{-1}.
\end{equation}
Thus, applying \cite[Proposition~1.13.5]{DF2022ESO1} to $\mathbf{A}^{-1}$ with $R = |\lambda|^{-1}$ and using $\|\mathbf{A}^{-1}\| = \|\mathbf{A}\|$, one obtains
\begin{align*}
\angle (V(\mathbf{A}), \mathbf{A}U(\mathbf{A}))
 = \angle(V(\mathbf{A}), S(\mathbf{A}^{-1}) ) \leq \frac{\pi}{2 |\lambda| \|\mathbf{A}\|},
\end{align*}
as promised.
\end{proof}

\end{appendix}

\bibliographystyle{abbrv}

\bibliography{refs}

\end{document}